\documentclass[14pt]{extarticle}

\usepackage[utf8]{inputenc}
\usepackage[english]{babel}

\usepackage{amsopn,amsmath,amssymb,amsthm}
\usepackage{graphicx}
\usepackage{subcaption}
\usepackage{mathtools}

\usepackage[pdfencoding=unicode, psdextra, citecolor=blue, urlcolor=blue,
  linkcolor=blue, colorlinks=true, bookmarksopen=true]{hyperref}

\usepackage[style=numeric,backend=biber,autolang=other,sorting=none]{biblatex}
\renewbibmacro{in:}{}
\usepackage[style=russian]{csquotes}

\usepackage{indentfirst}
\usepackage[top=1.5cm,bottom=1.5cm,left=1.5cm,right=1.5cm]{geometry}
\usepackage{comment}

\usepackage{graphicx}
\graphicspath{{pics/}}

\newcommand{\R}{\mathbb{R}}
\newcommand{\Z}{\mathbb{Z}}
\newcommand{\eqdef}{\stackrel{\mathrm{def}}{=}}
\newcommand{\ddt}{\dfrac{d}{dt}}
\newcommand{\const}{\mathrm{const}}
\newcommand{\polar}{\circ}
\newcommand{\SQ}{\mathbb{S}}

\DeclareMathOperator{\Int}{\mathrm{int}}
\DeclareMathOperator{\tg}{tg}

\DeclareMathOperator{\argmax}{\mathrm{arg\,max}}

\theoremstyle{plain}
\newtheorem{thm}{Theorem}
\newtheorem{prop}{Proposition}
\newtheorem{corollary}{Corollary}

\theoremstyle{definition}
\newtheorem{defn}{Definition}
\newtheorem{example}{Example}

\title{Convex trigonometry with applications to sub-Finsler geometry}
\author{L.\,V.~Lokutsievskiy}

\begin{document}

\maketitle

\begin{abstract}
	A new convenient method of describing flat convex compact sets is proposed. It generalizes classical trigonometric functions $\sin$ and $\cos$. Apparently, this method may be very useful for explicit description of solutions of optimal control problems with two-dimensional control. Using this method a series of sub-Finsler problems with two-dimensional control lying in an arbitrary convex set $\Omega$ is investigated. Namely, problems on the Heisenberg, Engel, and Cartan groups and also Grushin's and Martinet's cases are considered. A particular attention is paid to the case when $\Omega$ is a polygon.
\end{abstract}

\footnotetext[0]{This work was supported by the Russian Foundation for Basic Research under grants 17-01-00805 and 17-01-00809.}

\section{Introduction}

Let $\Omega$ be an arbitrary convex compact set in $\R^2$ with the origin in its interior, $0\in\Int\Omega$. If we have an optimal control problem where a 2-dimensional control $u=(u_1,u_2)\subset\R^2$ is restricted to $\Omega$, $u\in\Omega$, then Pontryagin's maximum principle usually states that the optimal control should move along the boundary of $\Omega$. In the case when $\Omega$ is a circle, this motion can be conveniently described in terms of trigonometric functions. However, if $\Omega$ is not a circle, then trigonometric functions are not the best choice. For example if $\Omega$ is a polygon, then explicit construction of optimal solutions involves painstaking considerations of all possible control jumps from one vertex to another. 

An entirely different approach is proposed in the present paper. We introduce new functions $\cos_\Omega$ and $\sin_\Omega$, which usually allow one to conveniently and explicitly describe the dynamics of a point along the boundary $\partial\Omega$ and to avoid cumbersome formulae. The functions  $\cos_\Omega$ and $\sin_\Omega$ coincide with the classical trigonometric functions $\cos$ and $\sin$ in the case when $\Omega$ is the unit circle centered at the origin. For the general case, (i) these functions inherit a lot of convenient properties of the classical functions $\cos$ and $\sin$  and (ii) they can be explicitly calculated  for a variety of particular sets $\Omega$. For example, the functions $\cos_\Omega$ and $\sin_\Omega$ are calculated completely below for the case when $\Omega$ is an arbitrary polygon. A connection with Jacobi elliptic functions is also computed for the case when $\Omega$ is an ellipse. The construction of the functions $\cos_\Omega$ and $\sin_\Omega$ involves directly the polar set $\Omega^\polar$. It appears that the properties of one pair of the functions $\cos$ and $\sin$ (obtained in the case when $\Omega$ is a circle) are inherited by two pairs of the generalized trigonometric functions $\cos_\Omega$, $\sin_\Omega$ and $\cos_{\Omega^\polar}$, $\sin_{\Omega^\polar}$.

We demonstrate the convenience of new functions by describing the dynamics of an optimal control in a series of sub-Finsler problems. Interest to these problems has increased in recent years also because of Gromov's theorem on groups with polynomial growth~\cite{GromovPolynomialGrowth}. Let us formulate these problems in the optimal control language. Consider the following time-optimal problems:

\[
	T\to\min
\]

\noindent assuming that the control $u=(u_1,u_2)$ is 2-dimensional and belongs to $\Omega$, $u\in\Omega$.

\begin{enumerate}
	\item On the Heisenberg group (the Dido problem with sub-Finsler length):
	\[
		\dot x_1=u_1,\quad \dot x_2=u_2,\quad \dot z=\frac12(x_1u_2-x_2u_1);\quad u\in\Omega.
	\]

	\item Grushin's problem:
	\[
		\dot x_1=u_1,\quad \dot x_2=x_1u_2;\quad u\in\Omega.
	\]

	\item Martinet's problem:
	\[
		\dot x_1=u_1,\quad \dot x_2=u_2,\quad\dot w=-\frac12x_2^2u_1;\quad u\in\Omega.
	\]

	\item On the Engel group:
	\begin{gather*}
		\dot x_1=u_1;\quad \dot x_2=u_2;\quad \dot z=\frac12(x_1u_2-x_2u_1);\\
		\dot w=-\frac12x_2^2u_1;\quad u\in\Omega.
	\end{gather*}

	\item On the Cartan group (the generalized Dido problem):
	\begin{gather*}
		\dot x_1=u_1,\quad \dot x_2=u_2,\quad \dot z=\frac12(x_1u_2-x_2u_1),\\
		\dot w_1=\frac12x_1^2u_2,\quad \dot w_2=-\frac12x_2^2u_1,\qquad u\in\Omega.\\
	\end{gather*}
	
\end{enumerate}

Each of the above problems defines a distance $\rho$ on the corresponding space in the standard way. However, $\rho$ is not a metric in general case but it is always a quasi-metric. Indeed (i)  $0\le\rho(P,Q)<\infty$ (in view of controlability), (ii) $\rho(P,Q)=0\Leftrightarrow P=Q$, and (iii) the triangle inequality $\rho(P,R)\le\rho(P,Q)+\rho(Q,R)$ is fulfilled, but there is no symmetry $\rho(P,Q)\ne\rho(Q,P)$ in general. Nevertheless $-\Omega\subset C\Omega$ for a constant $C\ge 1$ (since $\Omega$ is compact and $0\in\Int\Omega$), and consequently $\rho(P,Q)\le C\rho(Q,P)$. Thus we may call $\rho$ a quasi-metric. If in addition the set $\Omega$ is symmetric, $\Omega=-\Omega$, then $\rho$ becomes an actual metric. Let us note that any intrinsic left-invariant metric on a Lie group is a sub-Finser metric with a symmetric set of unit velocities $\Omega=-\Omega$ (see \cite[Theorem 2]{Berestovskii1988}).

The problem 1 was first considered by Busemann in 1947 in~\cite{Busemann} where closed geodesics were found and isoperemetric inequality was proved by Brunn-Minkowski's theory (and without using the Pontryagin maximum principle) for the set $\Omega$ being both convex and non-convex. Geodesics in problem 1 were completely found in~\cite{BerestovskiiHeisenberg} (with the help of the Pontryagin maximum principle) and optimal synthesis was constructed for the case when the set $\Omega$ is strictly convex. Also this problem was discussed in~\cite[\S 7]{Dmitruk}. Some of the above problems were considered in~\cite{BarilariBoscain} in the particular case when $\Omega$ is a square.

The dynamics of the optimal control in the problems 3, 4, and 5 is dictated by the following generalized pendulum equation
\[
	\ddot \theta^\polar = \sin_\Omega\theta,
\]

\noindent where $\theta$ and $\theta^\polar$ are generalized angles of the set $\Omega$ and $\Omega^\polar$ and they correspond to each other (see the next section for detailed definition). The structure of this equation's solutions is considered in details in the last Section~\ref{sec:pendulum} of the present paper.

The present paper is devoted to explicit integration of motion equations, so the author will not touch upon the questions of optimality in almost all cases.

\section{Convex trigonometry}

Let $\Omega\subset\R^2$ be a convex compact set and $0\in\Int\Omega$. The following definition of the functions $\cos_\Omega$ and $\sin_\Omega$ at first glance may cause a natural question ``why so?''. Nonetheless exactly this particular definition appears to be very convenient in solving all the above sub-Finsler problems, since it is agreed with affine transformations\footnote{In the same way one can define hyperbolic functions (see \cite{Shervatov}).}.

Denote by $\SQ(\Omega)$ the area of the set $\Omega$.

\begin{figure}[ht]
	\centering
	\includegraphics[width=0.5\textwidth]{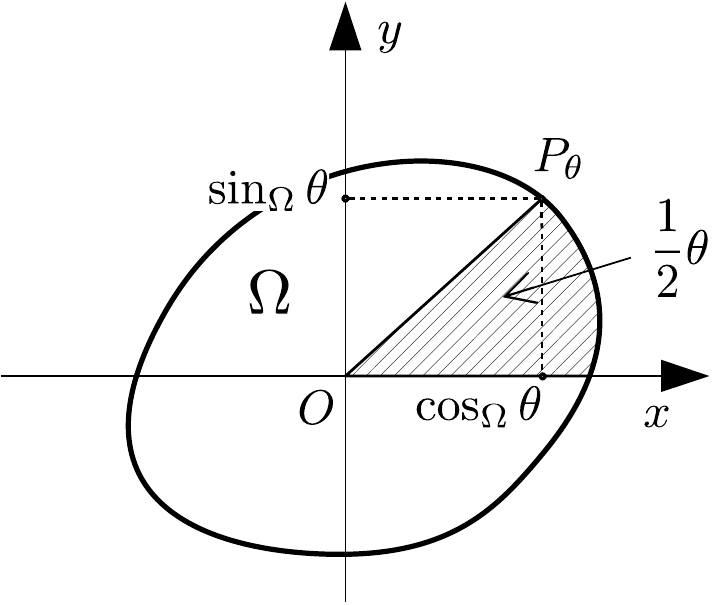}
	\caption{Definition of the generalized trigonometric functions $\cos_\Omega\theta$ and $\sin_\Omega\theta$ by the set $\Omega$.}
	\label{fig:cos_sin_def}
\end{figure}

\begin{defn}
	Let $\theta\in\R$ denote a generalized angle. If $0\le\theta<2\SQ(\Omega)$, then we choose a point $P_\theta$ on the boundary of $\Omega$ such that the area of the sector of $\Omega$ between the rays $Ox$ and $OP_\theta$ is $\frac12\theta$ (see Fig.~\ref{fig:cos_sin_def}). By definition $\cos_\Omega\theta$ and $\sin_\Omega\theta$ are the coordinates of~$P_\theta$. If the generalized angle $\theta$ does not belong to the interval $\big[0;2\SQ(\Omega)\big)$, then we define the functions $\cos_\Omega$ and $\sin_\Omega$ as periodic with period $2\SQ(\Omega)$; i.e., for $k\in\Z$ such that $\theta + 2\SQ(\Omega)k \in \big[0;2\SQ(\Omega)\big)$ we put
	\[
	\cos_\Omega\theta = \cos_\Omega(\theta + 2\SQ(\Omega)k);\qquad
	\sin_\Omega\theta = \sin_\Omega(\theta + 2\SQ(\Omega)k);\qquad
	P_\theta = P_{\theta+2\SQ(\Omega)k}.
	\]
\end{defn}

Note that all the properties of $\sin_\Omega$ and $\cos_\Omega$ listed below can be easily proved once the appropriate definition is given. The purpose of the present paper is to show the convenience of the new machinery in solving a series of optimal control problems with two-dimensional control. 

Obviously, $\sin_\Omega0=0$. If $\Omega$ is the unit circle centered at the origin, then the above definition produces the classical trigonometric functions. If $\Omega$ differs from the unit circle, then the functions $\cos_\Omega$ and $\sin_\Omega$, of course, differ from the classical functions $\cos$ and $\sin$. Nonetheless they inherit a lot of properties from the classical case.

We will use the polar set $\Omega^\polar$ together with the set $\Omega$:
\[
	\Omega^\polar = \{(p,q)\in\R^{2*}:px+qy\le 1\mbox{ for all }(x,y)\in\Omega\}\subset\R^2.
\]

\noindent The polar set $\Omega^\polar$ is (always) a convex and compact (as $0\in\Int\Omega$) set and $0\in\Int\Omega^\polar$ (as $\Omega$ is bounded). To avoid confusion we will assume that the set $\Omega$ lies in the plane with coordinates $(x,y)$ and that the polar set $\Omega^\polar$ lies in the plane with coordinates $(p,q)$.

Note that $\Omega^{\circ\circ}=\Omega$ by the bipolar theorem (see \cite[Theorem 14.5]{Rockafellar}). We can apply the above definition of the generalized trigonometric functions to the polar set $\Omega^\polar$ and an arbitrary angle $\psi\in\R$ to construct $\cos_{\Omega^\polar}\psi$ and $\sin_{\Omega^\polar}\psi$, which are the coordinates of the appropriate point $Q_\psi\in\partial\Omega^\polar$. From the definition of the polar set it follows that
\[
	\cos_\Omega\theta\cos_{\Omega^\polar}\psi + 
	\sin_\Omega\theta\sin_{\Omega^\polar}\psi \le 1.
\]

\begin{figure}[ht]
	\centering
	\includegraphics[width=0.5\textwidth]{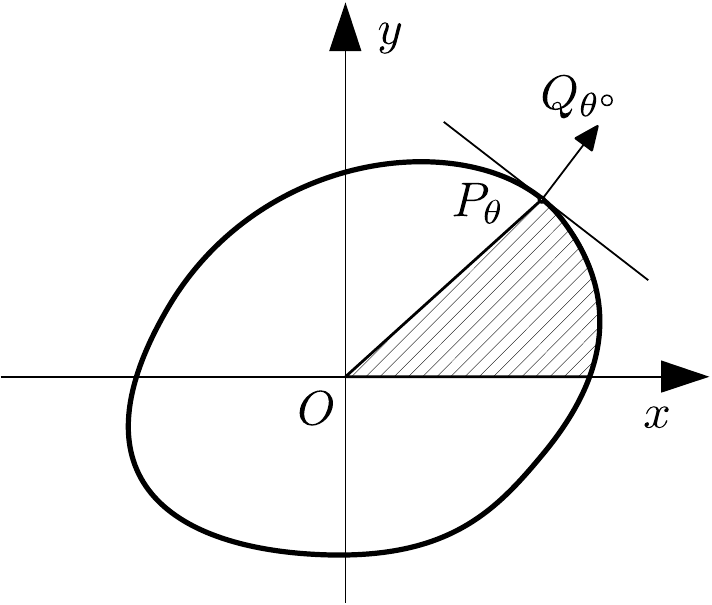}
	\caption{Correspondence $\theta\leftrightarrow\theta^\polar$.}
	\label{fig:thetha_to_theta_polar_pic}
\end{figure}

\begin{defn}
	We say that angles $\theta\in\R$ and $\theta^\polar\in\R$ correspond to each other and write $\theta\xleftrightarrow{\Omega}\theta^\polar$ if the supporting half-plane of $\Omega$ at $P_\theta$ is determined by the (co)vector~$Q_{\theta^\polar}$ (see fig. \ref{fig:thetha_to_theta_polar_pic}). When no confusing ensues we omit the symbol $\Omega$ over the arrow and write~$\theta\leftrightarrow\theta^\polar$.
\end{defn}

\begin{thm}
\label{thm:main_trig_thm}
	The definition of the correspondence of $\theta$ and $\theta^\polar$ is symmetric, i.e., $\theta\xleftrightarrow{\Omega}\theta^\polar$ is equivalent to $\theta^\polar\xleftrightarrow{\Omega^\polar}\theta$. Moreover, an analogue of the main Pythagorean identity holds:
	\begin{equation}
	\label{eq:main_trig}
		\theta\leftrightarrow\theta^\polar
		\quad\Longleftrightarrow\quad
		\cos_\Omega\theta\cos_{\Omega^\polar}\theta^\polar + 
		\sin_\Omega\theta\sin_{\Omega^\polar}\theta^\polar = 1.
	\end{equation}
\end{thm}

\begin{proof}
	Let us compute the generalized trigonometric functions in terms of the support function of the set $\Omega$:
	\[
		s_\Omega(p,q) = \sup_{(x,y)\in\Omega}(px+qy).
	\]

	\noindent The polar set can be written as $\Omega^\polar=\{(p,q):s_\Omega(p,q)\le 1\}$. The subdifferential \cite[\S\S 23,24]{Rockafellar} of the function $s_\Omega$ is
	\[
		\partial s_\Omega(p,q) = \argmax_{(x,y)\in\Omega}(px+qy).
	\]

	Thus the point $(x,y)\in\partial\Omega$ belongs to the subdifferential $\partial s_\Omega(p,q)$ if and only if the covector $(p,q)$ determines a supporting half-plane of the set $\Omega$ at the point $(x,y)$. If in addition $(p,q)\in\partial\Omega^\polar$, then
	\[
		(x,y)\in\partial s_{\Omega}(p,q)
		\quad\Leftrightarrow\quad
		1=s_\Omega(p,q)=px+qy.
	\]

	So $\theta\xleftrightarrow{\Omega}\theta^\polar$ if and only if the previous equality holds for $P_\theta=(x,y)$ and $Q_{\theta^\polar}=(p,q)$. Recall that
	\[
		x=\cos_\Omega\theta,\quad y=\sin_\Omega\theta,\quad
		p=\cos_{\Omega^\polar}\theta^\polar,\quad q=\sin_{\Omega^\polar}\theta^\polar.
	\]
	
	\noindent Consequently $\theta\xleftrightarrow{\Omega}\theta^\polar$ if and only if the equality in~\eqref{eq:main_trig} holds.
	
	Using the bipolar theorem it can be proved in the same way that the equality in~\eqref{eq:main_trig} is equivalent to the inverse correspondence $\theta^\polar\xleftrightarrow{\Omega^\polar}\theta$. Thus the correspondence $\leftrightarrow$ is symmetric.
\end{proof}

The correspondence $\theta\leftrightarrow\theta^\polar$ is not one-to-one in general. If the boundary of $\Omega$ has a corner at a point $P_\theta$, then the angle $\theta$ corresponds to the whole edge in $\Omega^\polar$ and vice versa, i.e., to any angle $\theta$ with $P_\theta$ on the same edge of $\Omega$ there corresponds one particular angle $\theta^\polar$ (up to $2\SQ(\Omega^\polar)\Z$), and the boundary of $\Omega^\polar$ has a corner at the point $Q_{\theta^\polar}$. Nonetheless it is natural to define a monotonic (multivalued and closed) function $\theta^\polar(\theta)$ that maps an angle $\theta$ to a maximal closed interval\footnote{Obviously, for any $\theta$ there exists an angle $\theta^\polar$ such that $\theta^\polar\leftrightarrow\theta$. This can be easily proved by the hyperplane separation theorem.} of angles $\theta^\polar$ such that $\theta^\polar\leftrightarrow\theta$. This function is quasiperiodic, i.e.,
\[
	\theta^\polar(\theta+2\SQ(\Omega)k) = \theta^\polar(\theta) + 2\SQ(\Omega^\polar)k
	\quad\mbox{with}\quad k\in\Z.
\]

If $\Omega$ is strictly convex, then the function $\theta^\polar(\theta)$ is strictly monotonic. If the boundary of $\Omega$ is $C^1$-smooth, then the function $\theta^\polar(\theta)$ is continuous.

\begin{thm}
\label{thm:derivatives}
	The functions $\cos_\Omega$ and $\sin_\Omega$ are Lipschitz continuous and have the left and right derivatives for all $\theta$, which coinside for a.e.\ $\theta$. Let us denote for short the whole interval between the left and right derivatives by the usual derivative stroke sign. Then for any $\theta\leftrightarrow\theta^\polar$, we have
	\[
		\cos'_\Omega\theta \ni -\sin_{\Omega^\polar}\theta^\polar
		\quad\mbox{ and }\quad
		\sin'_\Omega\theta \ni \cos_{\Omega^\polar}\theta^\polar.
	\]

	\noindent Moreover, for any $\theta$
	\begin{gather*}
		\cos'_\Omega\theta = \{-\sin_{\Omega^\polar}\theta^\polar\quad\mbox{for all}\quad \theta^\polar\leftrightarrow\theta\},\\
		\sin'_\Omega\theta = \{\cos_{\Omega^\polar}\theta^\polar\quad\mbox{for all}\quad \theta^\polar\leftrightarrow\theta\}.
	\end{gather*}

	The similar formulae hold for $\cos_{\Omega^\polar}'\theta^\polar$ and $\sin_{\Omega^\polar}'\theta^\polar$
\end{thm}

\begin{proof}
	Denote by $\pi_\Omega(\theta)$ the angle\footnote{This angle, of course, is defined up to the term $2\pi k$. We will naturally assume that $\pi_\Omega(0)=0$ and the function $\pi_\Omega$ is continuous. Thus, for example, $\pi_\Omega(2\SQ(\Omega)k)=2\pi k$ for $k\in\Z$.} between $Ox$ and $OP_\theta$, i.e.
	\begin{equation}
	\label{eq:pi_omega_theta}
		(\cos_\Omega\theta,\sin_\Omega\theta)\upuparrows(\cos\pi_\Omega(\theta),\sin\pi_\Omega(\theta)).
	\end{equation}

	\noindent The function $\pi_\Omega$ is monotone increasing and is bi-Lipschitz continuous with constants determined by the minimal and maximal distance from the origin to the boundary of $\Omega$. Let us choose the parametrization of the boundary $\partial\Omega$ by angle $\phi$,
	\[
		x(\phi) = r(\phi)\cos\phi,\quad y(\phi)=r(\phi)\sin\phi,
	\]

	\noindent where $r(\phi)$ is chosen such that $(x(\phi),y(\phi))\in\partial\Omega$, i.e. $r(\phi)=1/s_{\Omega^\polar}(\cos\phi,\sin\phi)$. The boundary $\partial\Omega$ becomes a Lipschitz curve in this parametrization. So the coordinates of points on $\partial\Omega$ are Lipschitz continuous functions of the parameter $\phi$ and consequently of the parameter $\theta$. Any Lipschitz continuous function is a.e.\ differentiable, and its one-side derivatives coinside at any function differentiability point.

	\begin{figure}[ht]
		\centering
		\includegraphics[width=0.45\textwidth]{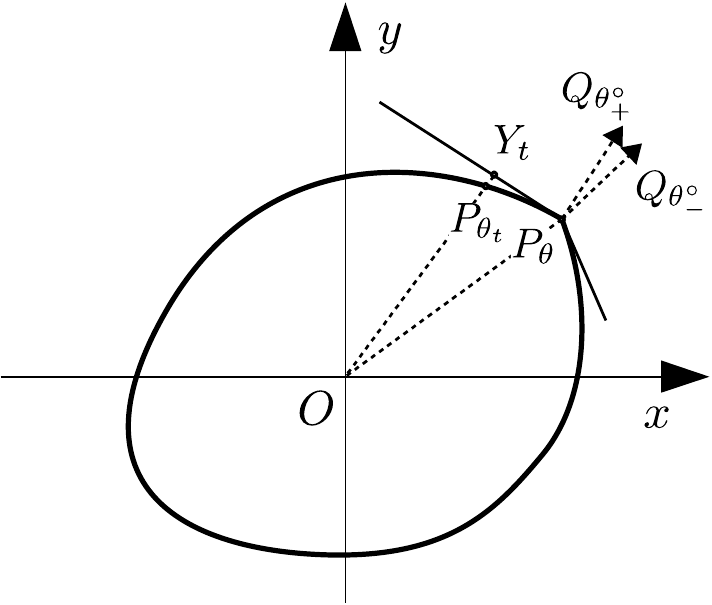}
		\caption{The points $P_\theta$, $Y_t$, $P_{\theta_t}$ and the tangent rays to $\partial\Omega$ determined by $\theta^\polar_\pm$.}
		
		\label{fig:diff_cos_sin}
	\end{figure}
	
	Now we want to compute the derivatives of the generalized trigonometric functions. Consider all supporting half-planes of $\Omega$ at $P_\theta$. They are determined by the covectors $(\cos_{\Omega^\polar}\theta^\polar,\sin_{\Omega^\polar}\theta^\polar)$ for all $\theta^\polar\leftrightarrow\theta$. The set of all angles $\theta^\polar$ corresponding to the given angle $\theta$ is a countable union of intervals (or points) of the form $[\theta^\polar_-+2\SQ(\Omega^\polar)k;\theta^\polar_++2\SQ(\Omega^\polar)k]$, $k\in\Z$. Thus one-side tangent rays to $\partial\Omega$ at $P_\theta$ have directions $(\mp\sin_{\Omega^\polar}\theta^\polar_{\pm},\pm\cos_{\Omega^\polar}\theta^\polar_\pm)$. Moreover, the angle $\theta^\polar_-$ corresponds to the clockwise tangent ray, and the angle $\theta^\polar_+$ corresponds to the counter-clockwise tangent ray (see Fig.~\ref{fig:diff_cos_sin}).

	Now let us compute the right derivatives of $\cos_\Omega$ and $\sin_\Omega$ (the left derivatives can be computed in a similar way). Denote by $Y_t$ for $t\ge 0$ the point on the tangent ray in $P_\theta$ corresponding to the angle $\theta^\polar_+$,
	\[
		Y_t = (\cos_\Omega\theta,\sin_\Omega\theta) + t(-\sin_{\Omega^\polar}\theta^\polar_+,\cos_{\Omega^\polar}\theta^\polar_+).
	\]

	The ray $OY_t$ crosses the boundary $\partial\Omega$ at the point $P_{\theta_t}$, which is determined by the angle $\theta_t$ (see Fig.~\ref{fig:diff_cos_sin}). Let us compute the coordinates of $P_{\theta_t}$ and the angle $\theta_t$ up to $o(t)$ terms. The point $Y_t$ moves along the ray, which is one-side tangent to $\partial\Omega$. So the coordinates of $P_{\theta_t}$ have the form
	\[
		(\cos_\Omega\theta_t,\sin_\Omega\theta_t) = P_{\theta_t} = Y_t + o(t) =
		P_\theta + t(-\sin_{\Omega^\polar}\theta^\polar_+,\cos_{\Omega^\polar}\theta^\polar_+) + o(t).
	\]
	
	For the same reasons, the generalized angle $\theta_t$ (which is the doubled area of the corresponding sector of $\Omega$) is equal up to the term $o(t)$ to the sum of $\theta$ and the doubled area of the triangle $OP_\theta Y_t$, i.e.,
	\[
		\theta_t = \theta + t [OP_\theta\times OY_t] + o(t) 
		=\theta + t(\cos_\Omega\theta\cos_{\Omega^\polar}\theta^\polar_+ + \sin_\Omega\theta\sin_{\Omega^\polar}\theta^\polar_+) +o(t)=\theta + t + o(t).
	\]

	\noindent (here we use Theorem~\ref{thm:main_trig_thm}).

	Hence $\ddt\theta_t=1$ for $t=+0$. So there exist right derivatives of functions $\cos_\Omega$ and $\sin_\Omega$ at $\theta$, and they have the following form
	\[
		\cos_\Omega'(\theta+0) = -\sin_{\Omega^\polar}\theta^\polar_+
		\quad\mbox{and}\quad
		\sin_\Omega'(\theta+0) = \cos_{\Omega^\polar}\theta^\polar_+.
	\]

	It remains to say that the derivatives of the functions $\cos_{\Omega^\polar}$ and $\sin_{\Omega^\polar}$ are computed automatically as $\Omega^{\polar\polar}=\Omega$ by the bipolar theorem.
\end{proof}

Let us note that any Lipschitz continuous function is a.e. differentiable. So if no confuse ensues we will write for short $\cos_\Omega'\theta=-\sin_{\Omega^\polar}\theta^\polar$ and $\sin_\Omega'\theta=\cos_{\Omega^\polar}\theta^\polar$ always meaning the result obtained in Theorem~\ref{thm:derivatives}.

It is easy to see that both the functions $\cos_\Omega$ and $\sin_\Omega$ have one interval of increasing and one interval of decreasing during their period. These two intervals can be separated by at most two intervals of constancy, which appear if $\Omega$ has edges parallel to the axes. Intervals of convexity and concavity can be also determined by the formulae of differentiation.

\begin{corollary}
	Each of the functions $\cos_\Omega$ and $\sin_\Omega$ is concave on any interval with non-positive values and is convex on any interval with non-negative values.
\end{corollary}

\begin{proof}
	Recall that the function $\theta^\polar(\theta)$ is monotone increasing. Since $\cos_\Omega'\theta=-\sin_{\Omega^\polar}\theta^\polar$, the function $\cos_\Omega\theta$ is convex (concave) on any interval on which the function $\sin_{\Omega^\polar}\theta^\polar$ is decreasing (increasing). These intervals are determined by the points on $\partial\Omega^\polar$ with horizontal supporting half-plane, i.e., $\cos_\Omega\theta=0$, the result required. Intervals of convexity and concavity of $\sin_\Omega\theta$ are constructed in a similar way.
\end{proof}

The more symmetries $\Omega$ has, the more symmetric $\cos_\Omega$ and $\sin_\Omega$ become. For example, the following result holds

\begin{prop}
	If $\Omega=-\Omega$, then
	\[
		\cos_\Omega(\theta+\SQ(\Omega)) = -\cos_\Omega\theta
		\quad\mbox{and}\quad
		\sin_\Omega(\theta+\SQ(\Omega)) = -\sin_\Omega\theta.
	\]
\end{prop}

The trigonometric addition formulae take the following form

\begin{prop}
\label{prop:angle_sum}
	Let $e^{i\phi}$ denote the rotation of $\R^2$ by the angle $\phi$ around the origin. Then
	\begin{align*}
		\cos_\Omega\theta \cos\phi -\sin_\Omega\theta\sin\phi &= \cos_{e^{i\phi}\Omega}(\theta+\pi^{-1}_\Omega(\phi)),\\
		\cos_\Omega\theta \sin\phi +\sin_\Omega\theta\cos\phi &= \sin_{e^{i\phi}\Omega}(\theta+\pi^{-1}_\Omega(\phi)).
	\end{align*}

\end{prop}

The last thing we need is the analogue of the polar change of coordinates:
\[
	\begin{cases}
		x=r\cos_\Omega\theta;\\
		y=r\sin_\Omega\theta.
	\end{cases}
\]

\noindent This change of variables is smooth in $r$ and Lipschitz continuous in $\theta$. Hence it has a.e.\ partial derivative with respect to $\theta$. The Jacobian matrix has the following form
\[
	\begin{pmatrix}
		x'_r & y'_r\\
		x'_\theta & y'_\theta
	\end{pmatrix}
	=
	\begin{pmatrix}
		\cos_\Omega\theta & \sin_\Omega\theta\\
		-r\sin_{\Omega^\polar}\theta^\polar & r\cos_{\Omega^\polar}\theta^\polar\\
	\end{pmatrix},
	\quad\mbox{where}\quad
	\theta^\polar \leftrightarrow\theta.
\]

\noindent Using the main trigonometric identity we see that the Jacobian is equal to~$r$.

Let us find the inverse change of variables $r(x,y)$ and $\theta(x,y)$. The radius can be found as follows
\[
	s_{\Omega^\polar}(x,y) = s_{\Omega^\polar}(r\cos_\Omega\theta,r\sin_\Omega\theta) =
	rs_{\Omega^\polar}(\cos_\Omega\theta,\sin_\Omega\theta) = r.
\]

The inverse Jacobian matrix for finding $\theta$ is computed as
\[
	\begin{pmatrix}
		r'_{x} & \theta'_{x}\\
		r'_{y} & \theta'_{y}\\
	\end{pmatrix}
	=
	\begin{pmatrix}
		\cos_{\Omega^\polar}\theta^\polar & -\frac1r\sin_\Omega\theta\\
		\sin_{\Omega^\polar}\theta^\polar & \frac1r\cos_\Omega\theta\\
	\end{pmatrix}.
\]

\noindent Thus for a.e.~$\theta$, we have
\begin{equation}
\label{eq:d_theta}
	d\theta(x,y) =  \frac{xdy-ydx}{s_{\Omega^\polar}^2(x,y)}.
\end{equation}

Now let us connect the point $(1,0)$ with the point $(\hat x,\hat y)$ (different from the origin) by an arbitrary curve $\gamma$ that does not contain the origin. The value  $\theta(\hat x,\hat y)$ is equal to the integral\footnote{Intervals of $\gamma$ where $\theta$ is constant do not affect the integral value, since $x\,dy-y\,dx=0$ on these intervals.}
\[
	\theta(\hat x,\hat y) = \int_\gamma \frac{xdy-ydx}{s_{\Omega^\polar}^2(x,y)}.
\]

\noindent Is it easy to see by Green's theorem that this integral is equal to the doubled area of the set that is swept by the radius vector of the point $(\frac{x}{r},\frac{y}{r})\in\partial\Omega$ while the point $(x,y)$ is moving along $\gamma$ (formula~\eqref{eq:d_theta} can also be obtained from this observation).

Note that in the classical case the angles are defined up to a summand $2\pi k$, $k\in\Z$. Here we have a similar situation: generalized angles are defined up to a summand $2\SQ(\Omega)k$, $k\in\Z$, which depends on the number of rotations around the origin that $\gamma$ makes.

Using the obtained formulae we will be able to completely compute in the next section the functions $\cos_\Omega$ and $\sin_\Omega$ in the case when $\Omega$ is an arbitrary polygon. Now let us compute these functions for some other useful cases.

\begin{example}
\label{ex:parametric_border}
	Assume that the boundary of $\Omega$ is defined parametrically,
	\[
		\partial\Omega=\{(x(t),y(t))\in\partial\Omega, t\in[0;t_0]\}.
	\]

	\noindent In this case for any $t$ there is defined a generalized angle $\theta(t)$ such that $P_{\theta(t)}=(x(t),y(t))$. Using~\eqref{eq:d_theta} we get\footnote{Of course the formula has this form, since $\theta$ is the doubled area of the correspond sector.}
	\[
		\theta(t)=\int_0^t\big[x(\tau)\dot y(\tau)-y(\tau)\dot x(\tau)\big]\,d\tau,
	\]

	\noindent since $s_{\Omega^\polar}(x(\tau),y(\tau))=1$. We have $\cos_\Omega\theta(t)=x(t)$ and $\sin_\Omega\theta(t)=y(t)$. Thus it remains to find the inverse function $t(\theta)$ and substitute it into $x(t)$ and $y(t)$.
\end{example}

\begin{example}
	If $\Omega$ is an ellipse with boundary $x(t)=a\cos t$, $y(t)=b\sin t$, then $\theta(t) = ab\,t$. So $\cos_\Omega\theta = a\cos\frac{\theta}{ab}$ and $\sin_\Omega\theta = b\sin\frac{\theta}{ab}$.

	The functions $\cos_\Omega$ and $\sin_\Omega$ suggested in the paper are closely related to the Jacobi elliptic functions in the case when $\Omega$ is an ellipse. For convenience we set $a=1$ and $b>1$. We put as always $m=1-1/b^2$ and $k^2=m$. Thus we obviously have
	\[
		\cos_\Omega\theta = \mathrm{cd\,}u=\frac{\mathrm{cn\,}u}{\mathrm{dn\,}u}
		\quad\mbox{and}\quad
		\sin_\Omega\theta = \mathrm{sd\,}u=\frac{\mathrm{sn\,}u}{\mathrm{dn\,}u}.
	\]

	\noindent Relations between the parameters $\theta$, $u$, and $\varphi$ can be easily found by dividing the previous equalities one by the other:
	\[
		\tg\left(\sqrt{1-k^2}\,\theta\right) = \sqrt{1-k^2}\,\mathrm{sc\,}u = \sqrt{1-k^2}\tg\varphi,
	\]

	\noindent (here $\varphi$ is the amplitude). The ratios of $d\theta$ to $du$ and $du$ to $d\varphi$ can be found by straightforward differentiation:	
	\[
		\frac{d\theta}{du}=\frac{du}{d\phi} = \frac{1}{\mathrm{dn\,}u}.
	\]
\end{example}

\begin{example}
	Let us compute the functions $\cos_\Omega$ and $\sin_\Omega$ under the assumption that the support function of the the polar set is known. Equivalently we may assume that $\Omega$ is given by an inequality for a non-negative positively homogeneous convex function: $\Omega=\{(x,y):s_{\Omega^\polar}(x,y)\le 1\}$. Let $\phi$ denote a classical angle and let $\theta(\phi)=\pi^{-1}_\Omega(\phi)$ be the corresponding generalized angle (see \eqref{eq:pi_omega_theta}).

	Let us parametrize the boundary $\partial\Omega$ by $\phi$, i.e.
	\[
		x(\phi)=\cos_\Omega\theta(\phi) = \frac{\cos\phi}{s_{\Omega^\polar}(\cos \phi,\sin\phi)}
		\quad\mbox{and}\quad
		y(\phi)=\cos_\Omega\theta(\phi) = \frac{\sin\phi}{s_{\Omega^\polar}(\cos \phi,\sin\phi)}.
	\]

	\noindent Denote  $r(\phi)=s_{\Omega^\polar}(\cos\phi,\sin\phi)$ for short. Then
	\begin{multline*}
		x\,dy-y\,dx = 
		\frac{\cos\phi}{r(\phi)}d\left(\frac{\sin\phi}{r(\phi)}\right) -
		\frac{\sin\phi}{r(\phi)}d\left(\frac{\cos\phi}{r(\phi)}\right) =\\
		=\frac{\cos^2\phi+\sin^2\phi}{r^2(\phi)} +
		\frac{\cos\phi\sin\phi}{r(\phi)}\left(d\frac{1}{r(\phi)}-d\frac{1}{r(\phi)}\right) =
		\frac{1}{r^2(\phi)}.
	\end{multline*}

	\noindent So
	\[
		\theta(\phi) = \pi^{-1}_\Omega(\phi) =
		\int_0^\phi \frac{d\psi}{s_{\Omega^\polar}^2(\cos\psi,\sin\psi)},
	\]
	
	\noindent and the values $\cos_\Omega\theta$ and $\sin_\Omega\theta$ can be found similarly to example \ref{ex:parametric_border}.
\end{example}

\section{Computation of the functions \texorpdfstring{$\cos_\Omega$}{cos} and \texorpdfstring{$\sin_\Omega$}{sin} for the polygon case.}
\label{sec:polyhedron}

Let $\Omega$ be a convex polygon in $\R^2$ with $n$ vertices and $0\in\Int\Omega$. Denote by $\hat P$ the intersection point of the ray $Ox$ with the boundary $\partial\Omega$. Next let us choose a counter-clockwise numbering of vertices $P_1$, $P_2$, $\ldots$, $P_n$ such that if $\hat P$ is a vertex, then $P_1=\hat P$, and if $\hat P$ belongs to the interior of an edge, then $P_1$ is a vertex of this edge lying in the upper half-plane (see Fig.~\ref{fig:polyhedron}). Denote by $x_k$ and $y_k$ the coordinates of $P_k=(x_k,y_k)$ and put $\hat P=(\hat x,0)$.

\begin{figure}[ht]
	\centering
	\includegraphics[width=0.5\textwidth]{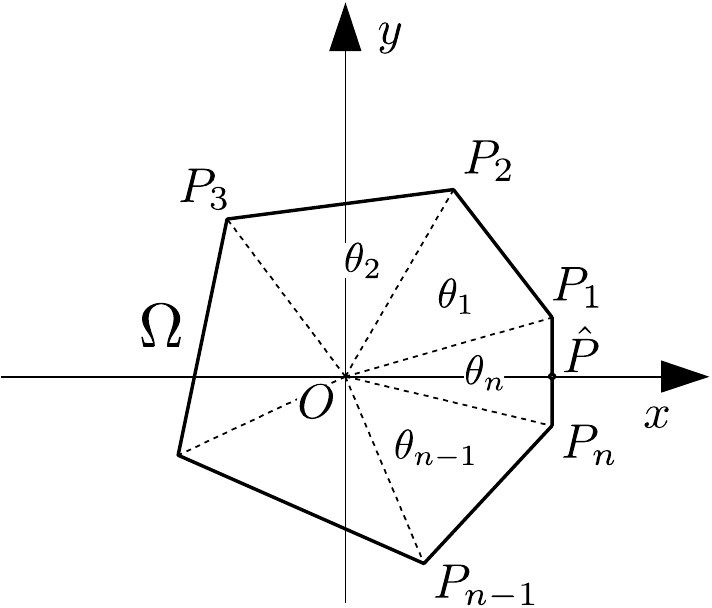}
	\caption{Computation of the functions $\cos_\Omega$ and $\sin_\Omega$ for the case when $\Omega$ is a polygon.}
	\label{fig:polyhedron}
\end{figure}

Let us find the generalized angles $\Theta_k$ corresponding to $P_k$, i.e. $P_k=P_{\Theta_k}$. We start with\footnote{Here $[\,\cdot\times\cdot\,]$ denotes the corresponding parallelogram's orientated area.} $\Theta_1$,
\[
	\Theta_1 = 2\SQ(\triangle O\hat PP_1) = [\hat P\times P_1] = \hat x y_1.
\]

The remaining $\Theta_k$ are also easy to find. If we denote by $\theta_k$ the doubled area of the triangle $\triangle OP_kP_{k+1}$ (i.e. $\theta_k=[P_k\times P_{k+1}]=x_ky_{k+1}-y_kx_{k+1}$), then
\[
	\Theta_{k+1} = 2\SQ(\triangle O\hat PP_1) + 2\sum_{j=1}^k\SQ(\triangle OP_jP_{j+1}) = \Theta_1+\sum_{j=1}^k\theta_j 
	= \Theta_k+\theta_k.
\]

Throughout we continue in natural way the numbering for indexes $k\le0$ or $k>n$. For example $P_{n+1}=P_1$, $\theta_{n+1}=\theta_{1}$, $\Theta_{n+1}=\Theta_1+2\SQ(\Omega)$ etc.

Now let us compute the functions $\cos_\Omega\theta$ and $\sin_\Omega\theta$. The period is easy to find,
\[
	2\SQ(\Omega) = \sum_{j=1}^{n}\theta_j.
\]

\begin{prop}
	If $\Omega$ is a convex polygon and $0\in\Int\Omega$, then each of the functions $\cos_\Omega$ and $\sin_\Omega$ is linear on any interval as the point $P_\theta$ moves along the same edge and does not pass through a vertex.
\end{prop}

\begin{proof}
	Chose an edge of $\Omega$ and a point $A_0=(x_0,y_0)$ lying on it. Let us make the point $(x,y)$ to move along this edge with a constant velocity $A_t=(x_0+at,y_0+bt)$. For any $t$ there is a generalized angle $\theta_t$ defined by the equation $P_{\theta_t}=A_t$. Since the area of the triangle $OA_0A_t$ depends linearly on $t$, $\theta_t$ is also a linear function of $t$. It remains to note that $\cos_\Omega \theta_t=x_0+at$ and $\sin_\Omega\theta_t=y_0+bt$.
\end{proof}

Hence the values of $\cos_\Omega\theta$ and $\sin_\Omega\theta$ at the vertices $P_k=P_{\Theta_k}$ determine these functions on the whole boundary $\partial\Omega$. Indeed, it is sufficient to prolong $\cos_\Omega\theta$ and $\sin_\Omega\theta$ piecewise-linearly from its values at $\Theta_k$ with period $2\SQ(\Omega)$. So $\cos_\Omega\Theta_k = x_k$ and $\sin_\Omega\Theta_k = y_k$. Therefore for any $\theta\in[\Theta_k;\Theta_{k+1}]$ we have
\[
	\cos_\Omega\theta = x_k + \frac{\theta-\Theta_k}{\theta_k}(x_{k+1}-x_k)
	\quad\mbox{and}\quad
	\sin_\Omega\theta = y_k + \frac{\theta-\Theta_k}{\theta_k}(y_{k+1}-y_k).
\]

\begin{figure}[ht]
	\centering
	\includegraphics[width=0.5\textwidth]{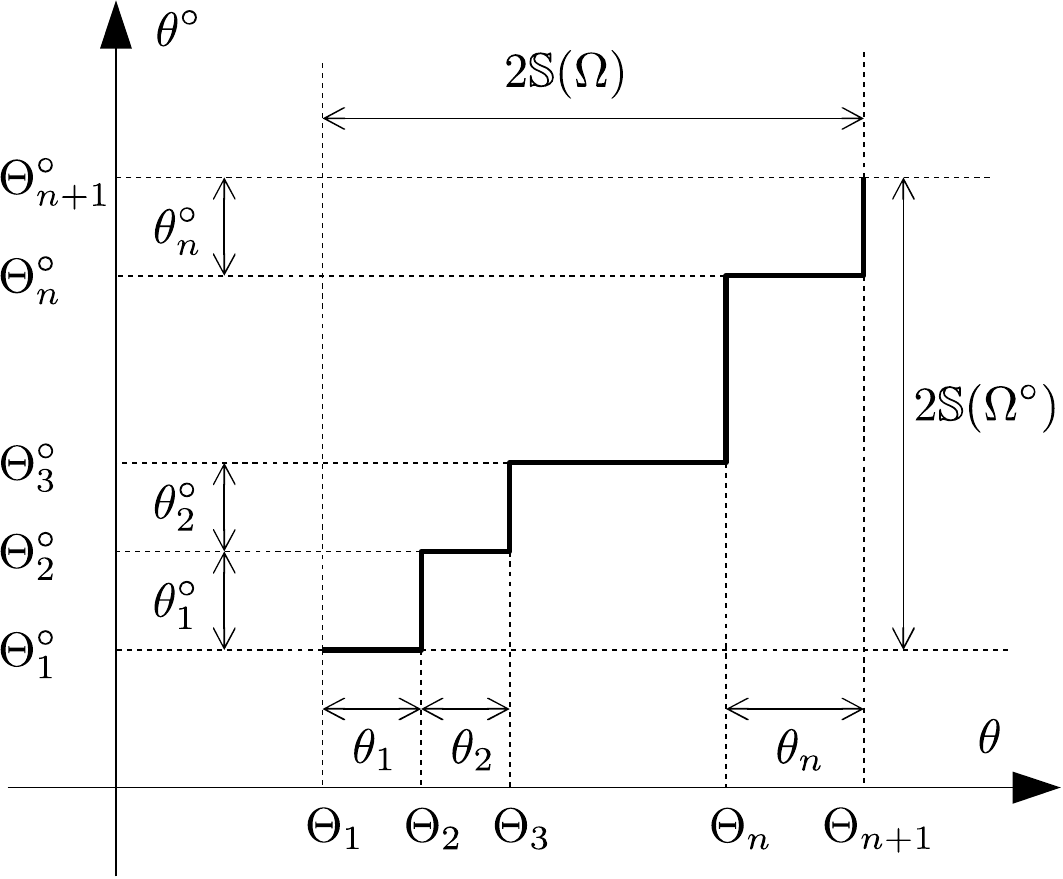}
	\caption{The dependence $\theta^\polar(\theta)$ for the case when $\Omega$ is a convex polygon.}
	\label{fig:theta_theta_polar_polyhedron}
\end{figure}

The last thing we need is to construct a transition to the polar set $\Omega^\polar$, which is also a polygon. The vertices of $\Omega$ become edges of $\Omega^\polar$ and vice versa. Denote by $Q_k$ the vertices of $\Omega^\polar$. So $Q_k$ can be found from the following conditions\footnote{We use the notation of angle brackets for the dot product $\langle(x,y),(p,q)\rangle=px+qy$, as usual.}
\[
	\langle Q_k,P_k \rangle = \langle Q_k,P_{k+1}\rangle = 1.
\]

\noindent Straightforward computation gives
\[
	Q_k = \frac{1}{\theta_k}e^{-i\pi/2}(P_{k+1}-P_k)
	=\frac{1}{\theta_k}\big(y_{k+1}-y_k,x_k-x_{k+1}\big).
\]

Now we are able to compute the generalized angles $\Theta_k^\polar$ for $\Omega^\polar$ using the same procedure we have used for $\Omega$. Indeed,
\[
	\Theta_{k+1}^\polar = \Theta_k^\polar + \theta_k^\polar,
\]
	
\noindent where
\[
	\theta_k^\polar = 2\SQ(OQ_kQ_{k+1}) = [OQ_k\times OQ_{k+1}] 
	= \frac{1}{\theta_k} + \frac{1}{\theta_{k+1}} - \frac{[P_k\times P_{k+2}]}{\theta_k\theta_{k+1}}.
\]

\noindent The function $\theta^\polar(\theta)$ has a stair-form structure as it is shown in Fig.~\ref{fig:theta_theta_polar_polyhedron}.

It is important to note here that the labeling of the vertices of $\Omega^\polar$ may not satisfy the rule of choosing of the first vertex, i.e. it may happen that the vertex $Q_1$ is not the ``first'' one. Namely, the edge $Q_nQ_1$ may not intersect the ray $Op$. Thus it remains to find the number $k$ of the ``first'' vertex $Q_k$ and to compute the area of the corresponding triangle $OQ_k\hat Q$ where $\hat Q=Op\cap \partial\Omega^\polar$.

Let $\hat Q$ be the intersection point of the ray $Op$ and the corresponding edge $Q_{k-1}Q_k$ of the polar set $\Omega^\polar$. It is very easy to find the point $\hat Q$ and the number $k$. If the polygon $\Omega$ has exactly one vertex $P_k$ with maximal first coordinate $x_k>0$, then $k$ is the required number and $\hat Q=(\frac{1}{x_k},0)$. If the polygon $\Omega$ has two vertices $P_k$ and $P_{k+1}$ with maximal first coordinate $x_k=x_{k+1}>0$ (i.e. $\Omega$ has a vertical edge lying in the right half-plane $x>0$), then $k$ is the required number and $\hat Q$ coincides with $Q_k$:  $\hat Q=Q_k = (\frac{1}{x_k},0)$.

\begin{figure}[ht]
	\centering
	\begin{subfigure}[b]{0.3\textwidth}
		\includegraphics[width=\textwidth]{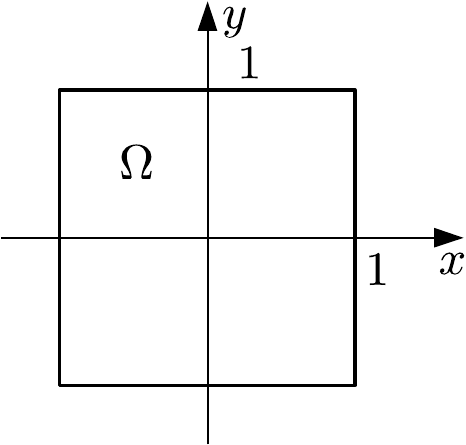}
	\end{subfigure}
	\ \ \ 
	\begin{subfigure}[b]{0.3\textwidth}
		\raisebox{11.5mm}{\includegraphics[width=\textwidth]{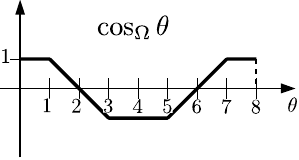}}
	\end{subfigure}
	\ \ \ 
	\begin{subfigure}[b]{0.3\textwidth}
		\raisebox{11.5mm}{\includegraphics[width=\textwidth]{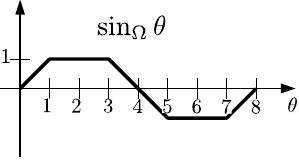}}
	\end{subfigure}\\
	\begin{subfigure}[b]{0.3\textwidth}
		\includegraphics[width=\textwidth]{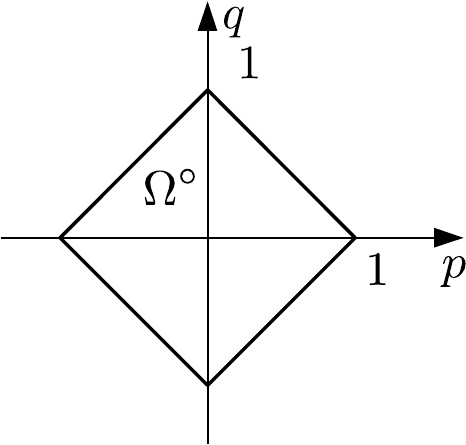}
	\end{subfigure}
	\ \ \ 
	\begin{subfigure}[b]{0.3\textwidth}
		\raisebox{11.5mm}{\includegraphics[width=\textwidth]{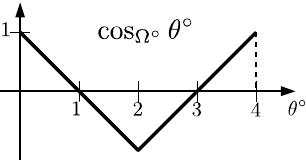}}
	\end{subfigure}
	\ \ \ 
	\begin{subfigure}[b]{0.3\textwidth}
		\raisebox{11.5mm}{\includegraphics[width=\textwidth]{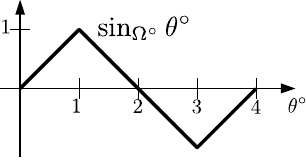}}
	\end{subfigure}
	\caption{Generalized trigonometric functions for the squares $\Omega$ and $\Omega^\polar$.}
	\label{fig:squares_cos_sin}
\end{figure}

\begin{example}
	Let $\Omega=\{|x|\le 1,|y|\le 1\}$. Then $\Omega^\polar=\{|x|+|y|\le 1\}$. The periods are $2\SQ(\Omega)=8$ and $2\SQ(\Omega^\polar)=4$. The graphs of the trigonometric functions for these sets are depicted in Fig.~\ref{fig:squares_cos_sin}. The function $\theta^\polar(\theta)$ is depicted in Fig.~\ref{fig:squares_theta_theta_polar}.
\end{example}

\begin{figure}[ht]
		\centering
		\includegraphics[width=0.4\textwidth]{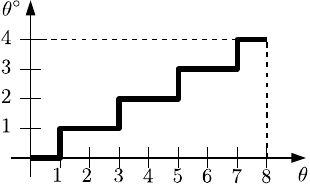}
		\caption{The dependence $\theta^\polar(\theta)$ for the square $\Omega=\{|x|\le1,|y|\le 1\}$.}
		\label{fig:squares_theta_theta_polar}
\end{figure}

\section{Heisenberg group}
\label{sec:Heisenberg}

In this section we will integrate equations of left-invariant sub-Finsler geodesic flows on the Heisenberg group with an arbitrary set $\Omega$ of unit velocities. So we will solve the following time-optimal problem
\[
	T\to\min,
\]
\[
	\dot x_1=u_1,\quad\dot x_2=u_2,\quad \dot z = \frac12(x_1u_2-x_2u_1),\qquad u=(u_1,u_2)\in\Omega.
\]

\noindent Here $\Omega\subset\R^2$ is an arbitrary compact convex set containing the origin in its interior, $0\in\Int\Omega$. We will use terminal constraints of the following general form:
\[
	x_1(0)=x_2(0)=z(0)=0,\qquad x_1(T)=x_1^0,\ x_2(T)=x_2^0,\ z(T)=z^0.
\]

A very detailed investigation of this problem for the case in which the initial and end points coincide can be found in~\cite{Busemann}. Full description of geodesics in this problem can be found in~\cite{BerestovskiiHeisenberg}. Here, in this section we will integrate the equations from Pontryagin's maximum principle in terms of the functions $\cos_\Omega$ and $\sin_\Omega$, which will greatly simplify computations.

Pontryagin's maximum principle reads as
\[
	\mathcal{H} = p_1 u_1 + p_2 u_2 + \frac12 q (x_1u_2-x_2u_1) = (p_1 - \frac12 qx_2)u_1 + (p_2+\frac12qx_1)u_2.
\]

\noindent Here $p_1$, $p_2$, and $q$ are the adjoint variables to $x_1$, $x_2$, and $z$, correspondingly. Maximum of $\mathcal{H}$ in $u\in\Omega$ can be written very conveniently in terms of the support function of the set~$\Omega$:
\[
	H=\max_{u\in\Omega}\mathcal{H} = s_\Omega\left(p_1 - \frac12 qx_2,\, p_2+\frac12qx_1\right).
\]

Let us denote for short the arguments of the support function by\footnote{Of course, $h_1$, $h_2$, and $h_3=q$ are left-invariant linear on fibers functions on the cotangent bundle of the Heisenberg group. Thus they form a basis of the Lie algebra and simultaneously they are coordinates on the Lie coalgebra. Obviously, with this choice of coordinates, the right-hand side of the equations from Pontryagin's maximum principle on $h_i$ will be separated from the rest of the system, and they will describe the flow of the Hamiltonian $\mathcal{H}$ on the Lie coalgebra w.r.t. the canonical Lie--Poisson bracket. However, these considerations are not essential for what follows.} $h_1 = p_1 - \frac12 qx_2$, $h_2=p_2+\frac12qx_1$. Obviously, $\ddt H=0$, i.e., $H=\const$. If $H=0$, then $h_1=h_2\equiv0$ and $q\ne0$ (since all the adjoint variables can not be equal to zero simultaneously), thus from $\dot h_1=\dot h_2\equiv 0$ we have $u_1=u_2\equiv0$, and the obtained trajectory is not optimal. So $H>0$. Consequently the point $(h_1,h_2)$ moves along the boundary of the polar set $\Omega^\polar$ stretched by $H$ times:
\[
	(h_1,h_2)\in H\partial\Omega^\polar.
\]

Now let us find the ``velocity'' of the point $(h_1,h_2)$. Substituting $\dot p_i=-\mathcal{H}'_{x_i}$ and $\dot q=-\mathcal{H}'_z$, we have
\[
	\dot h_1 = -q u_2,\quad\dot h_2 = qu_1,\quad\mbox{and}\quad \dot q=0.
\]

These equations can be conveniently integrated in terms of the generalized trigonometric functions. Put
\[
	h_1=r\cos_{\Omega^\polar}\theta^\polar\quad\mbox{and}\quad h_2=r\sin_{\Omega^\polar}\theta^\polar.
\]

\noindent We have $r=s_\Omega(h_1,h_2)=H=\const\ne 0$ as was mentioned above.

Let us compute the control function. Since $u=(u_1,u_2)\in\partial\Omega$, we have
\[
	u_1=\cos_\Omega\theta\quad\mbox{and}\quad u_2=\sin_{\Omega}\theta
\]

\noindent for some $\theta$. We claim that $\theta$ and $\theta^\polar$ correspond one to the other. Indeed,

\[
	u_1\cos_{\Omega^\polar}\theta^\polar + u_2\sin_{\Omega^\polar}\theta^\polar = \frac{1}{H}(h_1u_1+h_2u_2) = 1.
\]
\noindent Thus $\theta\leftrightarrow\theta^\polar$ by Theorem~\ref{thm:main_trig_thm}.

The last thing we need is to compute the derivative of $\theta^\polar$. According to~\eqref{eq:d_theta},
\[
	\dot\theta^\polar = \frac{h_1\dot h_2-h_2\dot h_1}{H^2} = 
	\frac{q}{H}(\cos_{\Omega^\polar}\theta^\polar \cos_\Omega\theta + \sin_{\Omega^\polar}\theta^\polar\sin_\Omega\theta) =
	\frac{q}{H}=\const.
\]

Hence an analogue of Kepler's law is fulfilled in any left-invariant sub-Finsler problem on the Heisenberg group: the radius vector of the point $(h_1,h_2)$ sweeps out equal areas during equal intervals of time:
\[
	\theta^\polar = \frac{q}{H}t + \theta_0^\polar;
\]

\noindent however, the point $(h_1,h_2)$ moves along the boundary of the stretched polar set~$H\partial\Omega^\polar$ rather than along an ellipse.

For example, using the above formulae it is very easy to find the first conjugate point $t_{conj}$: if $q\ne 0$, then the first point appears exactly at the instant when the point $(h_1,h_2)$ makes a full round, i.e. $\theta^\polar(t_{conj})=\theta^\polar_0+2\SQ(\Omega^\polar)$. So if $q\ne 0$, then
\[
	t_{conj} = \frac{2H}{q}\SQ(\Omega^\polar).
\]

Now we integrate the rest of Pontryagin's system. The control can be found by the angle $\theta$, but the correspondence $\theta\leftrightarrow\theta^\polar$ is not one-to-one. If the point $Q_{\theta^\polar}\in\partial\Omega^\polar$ (constructed by $\theta^\polar$) passes through a point where the boundary $\partial\Omega^\polar$ has a corner, then the generalized angle $\theta$ is not unique at this instant (it may take any value on the corresponding edge of $\Omega$). The boundary of a plane convex set, however, may have only finite or countable number of corner points (since the sum of angles of these corner points never exceed $2\pi$). Thus if the velocity $\frac{q}{H}$ is different from 0, then both the angle $\theta$ and the control $u$ are found uniquely by $\theta^\polar$ for a.e.~$t$.

Thus if $q\ne 0$, then
\[
	\dot x_1=\dot h_2/q\quad\mbox{and}\quad\dot x_2=-\dot h_1/q.
\]

\noindent From $x_1(0)=x_2(0)=0$ we obtain
\[
	x_1=\frac{H}{q}\left(\sin_{\Omega^\polar}\big(\frac{q}{H}t+\theta^\polar_0\big) - \sin_{\Omega^\polar}\theta^\polar_0\right)
	\quad\mbox{and}\quad
	x_2=-\frac{H}{q}\left(\cos_{\Omega^\polar}\big(\frac{q}{H}t+\theta^\polar_0\big) - \cos_{\Omega^\polar}\theta^\polar_0\right).
\]

\noindent Consequently the point $(x_1,x_2)$ moves along the shifted and rotated by $-\frac\pi2$ boundary of the polar set $\Omega^\polar$:
\[
	(x_1,x_2) \in \frac{H}{q}e^{-i\pi/2}\Omega^\polar +
	\frac{H}{q}(-\sin_{\Omega^\polar}\theta^\polar_0,\cos_{\Omega^\polar}\theta^\polar_0).
\]

Substituting $x_1$, $x_2$, $u_1$, and $u_2$ in the equation for $\dot z$ we get
\[
	\dot z = \frac{H}{2q}\big[
		1 - (\cos_{\Omega^\polar}\theta^\polar_0 \cos_{\Omega}\theta +
		\sin_{\Omega^\polar}\theta^\polar_0\sin_{\Omega}\theta)
	\big].
\]

Using Theorem~\ref{thm:derivatives} we have $\frac{d}{dt}\cos_{\Omega^\polar}\theta^\polar = -\dot\theta^\polar\sin_\Omega\theta$ and a similar formula for the derivative $\frac{d}{dt}\sin_{\Omega^\polar}\theta^\polar$. Since $\dot\theta^\polar=\frac{q}{H}=\const$, we have
\[
	z=\frac{H^2}{2q^2}\big(
		\theta^\polar + 
		\cos_{\Omega^\polar}\theta^\polar\sin_{\Omega^\polar}\theta^\polar_0 -
		\sin_{\Omega^\polar}\theta^\polar\cos_{\Omega^\polar}\theta^\polar_0
	\big).
\]

\noindent If the set $\Omega$ were the unit circle, then the difference in the written equation would turn into the sine of the difference. But this is not true in the general case.

Consider  the case $q=0$, i.e. $\theta^\polar=\theta^\polar_0=\const$. If the boundary $\partial\Omega^\polar$ is smooth at the point $Q_{\theta^\polar_0}$, then the angle $\theta=\theta_0$ corresponding to $\theta^\polar_0$ is unique, and we obtain the following straight line trajectory:
\[
	x_1=t\cos_\Omega\theta_0,\quad x_2=t\sin_\Omega\theta_0,\quad z=0.
\]

\noindent If the boundary $\partial\Omega^\polar$ has a corner at the point $Q_{\theta^\polar_0}$, then this point corresponds to a whole edge in $\partial\Omega$, and the control $(u_1,u_2)$ may take an arbitrary value lying on this edge at any instant. Usually this type of control is called \textit{singular on this edge of} $\Omega$. Using a rotation of $x_1,x_2$ we may assume that this edge is horizontal and lies at a height $\gamma>0$. If we restrict the control to this edge, we will get a new control system with $x_2=\gamma t$ and
\[
	\dot x_1=u_1\in[a;b],\quad \dot z = \frac\gamma2(x_1-tu_1).
\]

\noindent Any solution of this control system is optimal for the original system, since the second coordinate $x_2(T)$ at terminal end instant $T$ takes the maximal possible value $\gamma T$ among all trajectories of the original system\footnote{This idea belongs to the school of Yu.L.~Sachkov.}. Indeed, if $(u_1,u_2)\in\Omega$, then $u_2\le \gamma$, and consequently $x_2(T)\le\gamma T$.

Thus we reduce investigation of singular extremals on an edge of $\Omega$ to an investigation of a reachable set of a control system with one-dimensional control. Problems of this type can be solved by geometric version of Pontryagin's maximum principle (see~\cite[Theorem 12.1]{AgrachevSachkov}) and lie out of the main topic of this paper as the control becomes one-dimensional.

\section{The Grushin problem}

Consider the Grushin problem
\[
	T\to\min,
\]
\[
	\dot x_1=u_1,\quad \dot x_2 = x_1u_2,\quad u=(u_1,u_2)\in\Omega.
\]

\noindent Here $\Omega\subset\R^2$ is a convex compact set and $0\in\Int\Omega$, as always. We will use terminals constraints of the following general type:
\[
	x_1(0)=x_1^0,\ x_2(0)=x_2^0,\quad x_1(T)=x_1^1,\ x_2(T)=x_2^1.
\]

According to Pontryagin's maximum principle,
\[
	\mathcal{H} = p_1u_2 + p_2x_1u_2,
\]

\noindent where $p_1$ and $p_2$ are adjoint variables to $x_1$ and $x_2$. The maximum of $\mathcal{H}$ in $u\in\Omega$ has the following form:
\[
	H=\max_{u\in\Omega}\mathcal{H} = s_\Omega(p_1,p_2x_1).
\]

We claim that $H\ne 0$. Indeed, if $H=0$, then $p_1=p_2x_1\equiv 0$, which gives $p_1\equiv 0$ and $p_2=\const\ne 0$. Thus $x_1\equiv 0$ and consequently $x_2\equiv\const$. Obviously, the trajectory obtained is not optimal.

Similarly to Heisenberg's group, we denote arguments of the support function by $h_1=p_1$ and $h_2=p_2x_1$. Further, since $H=\const>0$, we have
\[
	(h_1,h_2) \in H\partial\Omega.
\]

Using the direct substitution $p_i=-\mathcal{H}'_{x_i}$ we obtain
\[
	\dot h_1 = -p_2u_2,\quad \dot h_2 = p_2u_1.
\]

\noindent Since $\dot p_2=0$, we arrive at the equations similar to the vertical part of the Pontryagin system for the Heisenberg case. They can be integrated in a similar manner by the generalized polar change of coordinates. Since $s_\Omega(h_1,h_2)=H$, the radius in this change if equal to $H$, i.e.,
\[
	h_1=H\cos_{\Omega^\polar}\theta^\polar\quad\mbox{and}\quad h_2=H\sin_{\Omega^\polar}\theta^\polar
\]

\noindent for an angle $\theta^\polar$. The derivative of $\theta^\polar$ can be found from~\eqref{eq:d_theta}:
\[
	\dot\theta^\polar = \frac{h_1\dot h_2-h_2\dot h_1}{s_\Omega^2(h_1,h_2)} =
	p_2\,\frac{h_1u_2+h_2u_2}{H^2} = \frac{p_2}{H}.
\]

\noindent Since $\dot p_2=0$, we obtain
\[
	\theta^\polar = \frac{p_2}{H}t + \theta^\polar_0,
\]

\noindent i.e.,
\[
	h_1=H\cos_{\Omega^\polar}\left(\frac{p_2}{H}t + \theta^\polar_0\right)
	\quad\mbox{and}\quad
	h_2=H\sin_{\Omega^\polar}\left(\frac{p_2}{H}t + \theta^\polar_0\right).
\]

Now we find the control. Using Theorem~\ref{thm:main_trig_thm} on the main trigonometric identity for the equation $h_1u_1+h_2u_2=H$ we obtain
\[
	u_1=\cos_\Omega\theta,\quad u_2=\sin_\Omega\theta
\]

\noindent for an angle $\theta\leftrightarrow\theta^\polar$.

Now let us integrate the remaining equations, i.e. find $x_1$ and $x_2$. If $p_2\ne 0$, then the angle $\theta$ (and thus the control) is uniquely determined by $\theta^\polar$ for a.e.~$t$. For $x_1$ we have
\[
	x_1 = \frac{h_2}{p_2} = \frac{H}{p_2}\sin_{\Omega^\polar}\theta^\polar.
\]

\noindent We will integrate the following equation to compute $x_2$:
\[
	\dot x_2 = x_1u_2 = \frac{H}{p_2}\sin_{\Omega^\polar}\theta^\polar\sin_\Omega\theta.
\]

\noindent To do this, we note that by Theorem~\ref{thm:derivatives} on derivatives of the trigonometric functions we have
\[
	(\cos_{\Omega^\polar}\theta^\polar\sin_{\Omega^\polar}\theta^\polar)'_{\theta^\polar}=
	\cos_{\Omega^\polar}\theta^\polar\cos_\Omega\theta - 
	\sin_{\Omega^\polar}\theta^\polar\sin_\Omega\theta =
	1-2\sin_{\Omega^\polar}\theta^\polar\sin_\Omega\theta.
\]

\noindent Thus from $\dot\theta^\polar=p_2/H$ we have
\[
	x_2 = x_2^0 + \frac{H^2}{2p_2^2} \big(
		\theta^\polar -  \cos_{\Omega^\polar}\theta^\polar\sin_{\Omega^\polar}\theta^\polar
	\big).
\]

Now consider the case $p_2=0$. In this case $\theta^\polar=\theta^\polar_0=\const$ and the point $(h_1,h_2)$ is not moving. If the boundary $\partial\Omega^\polar$ is smooth at the corresponding point, then the angle $\theta_0\leftrightarrow\theta^\polar_0$ is unique, and we obtain the following straight line trajectory:
\[
	x_1 = x_1^0 + t\cos_\Omega\theta_0,\quad x_2 = x_2^0 + t\sin_\Omega\theta_0.
\]

\noindent If the boundary $\partial\Omega^\polar$ has a corner at the corresponding point, then the control $(u_1,u_2)$ may take arbitrary values on the corresponding edge of $\Omega$, and we obtain a singular extremal on this edge. Each point of this edge can be represented in the form
\[
	u_1(v) = \frac{\cos_{\Omega^\polar}\theta^\polar_0 - v\sin_{\Omega^\polar}\theta^\polar_0}{\cos^2_{\Omega^\polar}\theta^\polar_0 + \sin^2_{\Omega^\polar}\theta^\polar_0},
	\quad
	u_2(v) = \frac{\sin_{\Omega^\polar}\theta^\polar_0 + v\cos_{\Omega^\polar}\theta^\polar_0}{\cos^2_{\Omega^\polar}\theta^\polar_0 + \sin^2_{\Omega^\polar}\theta^\polar_0}.
\]

\noindent Here $v\in[a,b]$ is a one-dimensional parameter. So we obtain the following control system that is affine in the one dimensional control $v$:
\[
	\dot x_1 = u_1(v),\quad \dot x_2 = x_1 u_2(v).
\]

\noindent Its reachable set can be found (similarly to the Heisenberg case) by the geometric version of the Ponryagin maximum principle.

\section{The Cartan group}
\label{sec:Cartan}

The left-invariant sub-Riemannian\footnote{I.e., a problem where the set $\Omega$ is a circle centered at the origin.}  problem on the Cartan group has a much reacher dynamics than the problem on the Heisenberg group (see~\cite{SachkovDido}). Sub-Riemannian (normal) geodesics are Euler elasticae and they are determined by the following inversed pendulum equation
\[
	\ddot\theta=\sin\theta.
\]

\noindent Sub-Finsler dynamics on Cartan's group is also much reacher than sub-Finsler dynamics on Heisenberg's group.

From the Hamiltonian point of view the process of integrating of any left-invariant sub-Finsler problem on Cartan's group with an arbitrary set $\Omega$ is reduced to the problem of finding solutions of a Hamitonian system on the corresponding Lie coalgebra. This coalgebra is five-dimensional and has 3 Casimirs; therefore any (smooth) Hamiltonian system on it can be integrated in quadratures.

However, there are two obstacles in this way while integrating equations of Pontryagin's maximum principle. First, the Hamiltonian system of the maximum principle is not smooth and has singular extremals, and, second, we need formulae that are convenient to work with even for the normal case.

So,
\[
	T\to\min,
\]
\[	
	\dot x_1=u_1,\ \dot x_2=u_2,\ \dot z=\frac12(x_1u_2-x_2u_1),\ \dot w_1 = \frac12x_1^2u_2,\ \dot w_2=-\frac12x_2^2u_1,
\]
\[
	u=(u_1,u_2)\in\Omega.
\]

\noindent According to Pontryagin's maximum principle we have
\[
	\mathcal{H}=p_1u_1+p_2u_2 + \frac{r}{2}(x_1u_2-x_2u_1) + \frac{q_1x_1^2}{2}u_2 - \frac{q_2x_2^2}{2}u_1.
\]

\noindent Here $p_1$, $p_2$ are the adjoint variables to $x_1$, $x_2$; $r$ to $z$; and $q_1$, $q_2$ to $w_1$, $w_2$, respectively.

Similarly to the Heisenberg case, the maximum of $\mathcal{H}$ in $u$ is expressed in terms of the support functions of the set $\Omega$:
\[
	H=\max_{u\in\Omega}\mathcal{H} = s_\Omega\Big(p_1-\frac{r}{2}x_2-\frac{q_2}{2}x_2^2,\,p_2+\frac{r}{2}x_1+\frac{q_1}{2}x_1^2\Big).
\]

Let us denote the arguments of $s_\Omega$ by $h_1$, $h_2$ as usual. Using classical notation\footnote{I.e. $h_1$, $h_2$, $h_3$, and $h_4$ are coordinates on the Lie coalgebra.} we get
\[
	\begin{cases}
		\dot h_1=-h_3u_2\\
		\dot h_2=h_3u_1\\
		\dot h_3=h_4u_1+h_5u_2\\
		\dot h_4=\dot h_5=0,
	\end{cases}
\]

\noindent where $h_3 = r+q_1x_1+q_2x_2$, $h_4=q_1$, and $h_5=q_2$. The Casimirs\footnote{On the Lie coalgebra w.r.t. the Lie--Poisson bracket.} are as follows:
\[
	\dot h_4=\dot h_5=\dot C=0,\ \mbox{where}\ C=\frac12h_3^2 + h_5h_1-h_4h_2.
\]

Obviously, $\ddt H=0$, i.e., $H=\const\ge 0$. Consider the case $H=0$. We obtain $h_1=h_2\equiv0$. If $h_3(\tau)\ne 0$ for some $\tau$, then $h_3\ne 0$ in a neighborhood of $\tau$, and thus from the identity $\dot h_1=\dot h_2\equiv 0$ we have $u_1=u_2=0$ in the neighborhood of $\tau$, i.e., $\dot h_3=0$ in the neighborhood of $\tau$. Consequently being continuous the function $h_3(t)$ is constant and thus $u_1=u_2\equiv 0$. The trajectory obtained is not optimal. So necessarily $h_3\equiv 0$. Thereby $(u_1,u_2)\perp(h_4,h_5)$ for all $t$, and we get a singular extremal, which is similar in all left-invariant sub-Finsler problems on the Cartan group independently on the set $\Omega$. Only the velocity of moving along this singular extremal depends on $\Omega$. It remains to say that sub-Riemannian singular extremals on the Cartan group are described in details in~\cite[Section~3.3]{SachkovDido}.

Assume that $H>0$. Since $H=s_\Omega(h_1,h_2)$, we see again that the vector $(h_1,h_2)$ moves along the boundary of the stretched (by $H>0$ times) polar set
\[
	(h_1,h_2)\in H\partial\Omega^\polar.
\]

\noindent The motion along the boundary $H\partial\Omega^\polar$ is not uniform w.r.t. the generalized angle (in contrast to the Heisenberg group). Let us find the dynamics of the generalized angle $\theta^\polar$ used in the following polar change of coordinates:
\[
	h_1=H\cos_{\Omega^\polar}\theta^\polar\quad\mbox{and}\quad h_2=H\sin_{\Omega^\polar}\theta^\polar.
\]

Since $h_1u_1+h_2u_2=H$, using Theorem~\ref{thm:main_trig_thm} we get
\[
	u_1=\cos_\Omega\theta\quad\mbox{and}\quad u_2=\sin_\Omega\theta
\]

\noindent for some $\theta\leftrightarrow\theta^\polar$. Thus the dynamics of the control is determined by the dynamics of~$\theta^\polar$. According to~\eqref{eq:d_theta} we have
\[
	\dot\theta^\polar = \frac{h_1\dot h_2-h_2\dot h_1}{s_\Omega^2(h_1,h_2)} = 
	\frac{H\cos_{\Omega^\polar}\theta^\polar\cos_\Omega\theta + H\sin_{\Omega^\polar}\theta^\polar\sin_\Omega\theta}{H^2}h_3 = \frac{h_3}{H}.
\]

Let us use the addition formula from Proposition~\ref{prop:angle_sum} to compute $\dot h_3$. Denote $q=\sqrt{h_4^2+h_5^2}$, i.e., $h_4=-q\sin\phi_0$ and $h_5=q\cos\phi_0$ for some $\phi_0\in\R$. Then
\[
	\dot h_3 = q(-\sin\phi_0\cos_\Omega\theta + \cos\phi_0\sin_\Omega\theta) =q\sin_{e^{-i\phi_0}\Omega}(\theta-\pi_\Omega^{-1}(\phi_0)).
\]

Let us assume that $\phi_0=0$. If it is not true we can always rotate all the considered objects by the angle $\phi_0$, i.e., redefine $\tilde\theta^\polar=\theta^\polar-\pi_{\Omega^\polar}^{-1}(\phi_0)$, $\theta=\theta-\pi_\Omega^{-1}(\phi_0)$, and $\tilde\Omega=e^{-i\phi_0}\Omega$. So in the case $h_4=0$ and $h_5=q$ we get the following equation:
\begin{equation}
\label{eq:generalized_pendulum_H}
	\ddot \theta^\polar = \frac{q}{H}\sin_\Omega\theta.
\end{equation}

\noindent The trajectories of the maximum principle in the plane $(x_1,x_2)$ obey this equation and generalize Euler's elasticae:
\[
	\dot x_1=\cos_\Omega\theta\quad\mbox{and}\quad\dot x_2=\sin_\Omega\theta.
\]

In the simplest case $q=0$ (i.e. $h_4=h_5=0$), we have $\ddot\theta^\polar=0$. Thus we get the uniform dynamics by Kepler's law similar to the Heisenberg group. In the case $q>0$ we get a fundamentally different equation, which generalizes the equations of the inversed\footnote{The inversed pendulum, of course, turns into a normal one if we choose $\phi_0=\pi$ instead of $\phi_0=0$.} physical pendulum.

In the case of sub-Riemannian problem on the Cartan group the dynamics of the control is determined by the classical pendulum equation, and the Casimir $C$ surprisingly determines the energy of this pendulum. This phenomenon is preserved in the sub-Finsler case. Indeed, denote
\[
	\mathbb{H}\eqdef \frac{C}{H} = \frac12 H\,{{}\dot\theta^\polar}^2 + q\cos_{\Omega^\polar}\theta^\polar.
\]

\noindent Then, obviously, $\dot{\mathbb{H}}=0$. Moreover, if we choose $h_3=H\dot\theta^\polar$ as a adjoint variable to $\theta^\polar$, then we receive the following Hamiltonian equations:
\[
	\dot\theta^\polar = \mathbb{H}'_{h_3} = \frac{h_3}{H}
	\quad\mbox{and}\quad
	\dot h_3 = -\mathbb{H}'_{\theta^\polar} = -(q\cos_{\Omega^\polar}\theta^\polar)'_{\theta^\polar} = q\sin_\Omega\theta.
\]

So integration of the equations of the maximum principle is reduced to solving equation~\eqref{eq:generalized_pendulum_H}. Note that the constant $\frac{q}{H}>0$ is not important in~\eqref{eq:generalized_pendulum_H}, since it can be eliminated by the following stretching of time: $t\mapsto\sqrt{H/q}\,t$. The structure of solutions of equation~\eqref{eq:generalized_pendulum_H} is considered separately in the last section, since the two remaining sub-Finsler problems are also reduced to this equation.

\section{Martinet's problem and the Engel group}
\label{sec:Martine_Engel}

Extremals in the both sub-Finsler problems for the case $H>0$ are described (similarly to the Cartan group) by the generalized pendulum equations~\eqref{eq:generalized_pendulum_H}. In this section we will deduce the motion equations as short as possible, since they are very similar to equations computed for the Cartan group.

Let us start with the Martinet problem
\[
	T\to\min,
\]
\[
	\dot x_1=u_1,\quad \dot x_2=u_2,\quad\dot z=-\frac12x_2^2u_1,
	\qquad
	(u_1,u_2)\in\Omega\subset\R^2.
\]

\noindent The set $\Omega$ is a convex compact set with the origin lying in its interior as always.

Pontryagin's maximum principle gives
\[
	\mathcal{H} = p_1 u_1 + p_2 u_2 - \frac{q}{2}x_2^2u_1=\const\ge 0. 
\]

Let us again describe the maximum of the Pontryagin function in terms of the support function of the set $\Omega$:
\[
	H=\max_{(u_1,u_2)\in\Omega}\mathcal{H} = s_\Omega(p_1 - \frac{q}{2}x_2^2,p_2)=\const.
\]

\noindent The case $H=0$ leads to singular extremals that coincide (up to the motion speed) with singular extremals in the sub-Riemannian case (see~\cite[Chapter~3]{Montgomery}).

Let $H>0$. Denote the arguments of the support function by $h_1=p_1-\frac{q}{2}x_2^2$ and $h_2=p_2$. Obviously, we get a motion along the boundary of the stretched polar set:
\[
	(h_1,h_2)\in H\partial\Omega^\polar.
\]

Using direct substitution we get
\[
	\dot h_1=-h_3u_2,
	\qquad
	\dot h_2=h_3u_1,
\]

\noindent where $h_3=qx_2$ and
\[
	\dot h_3 = qu_2,\quad q=\const.
\]

Denote
\[
	h_1= H\cos_{\Omega^\polar}\theta^\polar
	\quad\mbox{and}\quad
	h_2= H\sin_{\Omega^\polar}\theta^\polar.
\]

\noindent Then using Pontryagin's maximum principle and Theorem~\ref{thm:main_trig_thm} we obtain
\[
	u_1=\cos_\Omega\theta
	\quad\mbox{and}\quad
	u_2=\sin_\Omega\theta
\]

\noindent for some $\theta\leftrightarrow\theta^\polar$.

From~\eqref{eq:d_theta} we have
\[
	\dot\theta^\polar = \frac{h_1\dot h_2-h_2\dot h_1}{s_\Omega^2(h_1,h_2)} = \frac{h_3}{H}.
\]

\noindent It remains to note that $\dot h_3=qu_2=q\sin_\Omega\theta$, and consequently we arrive at the generalized pendulum equation $\ddot\theta^\polar=\frac{q}{H}\sin_\Omega\theta$, which coincides with~\eqref{eq:generalized_pendulum_H} (if $q<0$, then we can rotate the set $\Omega$ by $\pi$ and put $\tilde\Omega=-\Omega$ in similar to \S\ref{sec:Cartan} way).

\medskip

Let us now investigate the Engel group:
\[
	T\to\min,
\]
\[
	\dot x_1=u_1,\quad \dot x_2=u_2,\quad \dot z=\frac12(x_1u_2-x_2u_1),\quad
	\dot w=-\frac12x_2^2u_1,
	\qquad
	(u_1,u_2)\in\Omega.
\]

Using Pontryagin's maximum principle we get
\[
	\mathcal{H}=p_1u_1 + p_2u_2 + \frac{r}{2}(x_1u_2-x_2u_1) - \frac{q}{2}x_2^2u_1.
\]

The maximum of the Pontryagin function has the following form
\[
	H = \max_{(u_1,u_2)\in\Omega}\mathcal{H} = s_\Omega(p_1 - \frac{r}2x_2-\frac{q}{2}x_2^2,p_2+\frac{r}2x_1)=\const\ge 0.
\]

The case $H=0$ leads to singular extremals that coincide with singular extremals in the sub-Riemannian problem on the Engel group. A detailed descriptions of them can be found in~\cite[Section~4]{SachkovArdentovEngel}.

Let $H>0$. For the arguments $h_1=p_1 - \frac{r}2x_2-\frac{q}{2}x_2^2$ and $h_2=p_2+\frac{r}2x_1$ of the support function we have $(h_1,h_2)\in H\partial\Omega^\polar$ and
\[
	\dot h_1 = -h_3u_2,\qquad \dot h_2 = h_3u_1,
\]

\noindent where $h_3=qx_2+r$ and 
\[
	\dot h_3=qu_2,\quad q=\const.
\]

\noindent Thus using a similar substitution
\[
	h_1= H\cos_{\Omega^\polar}\theta^\polar
	\quad\mbox{and}\quad
	h_2= H\sin_{\Omega^\polar}\theta^\polar,
\]
\[
	u_1=\cos_\Omega\theta
	\quad\mbox{and}\quad
	u_2=\sin_\Omega\theta,
\]

\noindent where $\theta\leftrightarrow\theta^\polar$, we get the same generalized pendulum equation $\ddot\theta^\polar=\frac{q}{H}\sin_\Omega\theta$.

\section{Phase portraits of generalized pendulums defined by arbitrary sets \texorpdfstring{$\Omega$}{Ω}.}
\label{sec:pendulum}

In this section we will describe the phase portrait of the following equation
\begin{equation}
\label{eq:pendulum}
	\ddot\theta^\polar=\sin_\Omega\theta
\end{equation}

\noindent for an arbitrary convex compact set $\Omega$ that contains the origin in its interior, $0\in\Int\Omega$. A function $\theta^\polar(t)$ is called a solution of equation~\eqref{eq:pendulum} if it is  Lipschitz continuous together with its derivative and there exists a function $\theta(t)$ such that the correspondence $\theta(t)\leftrightarrow\theta^\polar(t)$ and equation~\eqref{eq:pendulum} hold for a.e.~$t$. In other words, equation~\eqref{eq:pendulum} can be considered as the following differential inclusion:
\[
	\ddot\theta^\polar\in\{\sin_\Omega\theta,\mbox{ for all }\theta\leftrightarrow\theta^\polar\}.
\]

In the present section we will describe the topological structure of the phase portrait of equation~\eqref{eq:pendulum}. It is very similar to structure of the phase portrait of the classical (inversed) pendulum equation $\ddot\theta=\sin\theta$; however, it may have one very important difference in the energy level of separatrix solutions. We will also integrate explicitly the pendulum equation~\eqref{eq:pendulum} in the case when $\Omega$ is a convex polygon.

Equation~\eqref{eq:pendulum} is invariant under the shift of $\theta^\polar$ by the period $2\SQ(\Omega^\polar)$. So it is sufficient to construct the phase portrait on the cylinder $\theta^\polar\in\R/2\SQ(\Omega^\polar)\Z$ and $\dot\theta^\polar\in\R$.

Technically, equation~\eqref{eq:pendulum} may have a non-unique solution, since the right-hand side contains a whole set of possible velocities $\sin_\Omega\theta$ for all $\theta\leftrightarrow\theta^\polar$ (see~\cite{FilippovDRHS}). Nonetheless we will see in what follows that there are no more than two points of non-uniqueness.

We already know that the energy is given by the following formula
\[
	\mathbb{H}=\frac12{{}\dot\theta^\polar}^2 + \cos_{\Omega^\polar}\theta^\polar.
\]

\noindent We claim that the energy is constant along solutions. Indeed, using equation~\eqref{eq:pendulum} and Theorem~\ref{thm:derivatives} we get $\dot{\mathbb{H}}=\dot\theta^\polar\sin_\Omega\theta - \dot\theta^\polar\sin_\Omega\theta_1$ for a.e.~$t$ and for some $\theta_1\leftrightarrow\theta^\polar$. Thus if $\dot\theta^\polar(t_0)=0$, then $\dot{\mathbb{H}}(t_0)=0$, and if $\dot\theta^\polar(t_0)\ne 0$, then the correspondence $\theta\leftrightarrow\theta^\polar$ is one-to-one for a.e.~$t$ in a neighborhood of $t_0$, and hence $\theta(t)=\theta_1(t)$ for a.e.~$t$ in this neighborhood.

From the constancy of $\mathbb{H}$ it immediately follows that if $\dot\theta^\polar\ne0$ for some $t_0$, then the solution is unique in a neighborhood of $t_0$, since it can be found by straightforward integration of a Lipschitz continuous function,
\[
	\int\sqrt2dt = \mathrm{sgn}\,\dot\theta^\polar(t_0)
	\int\frac{d\theta^\polar}{\sqrt{\mathbb{H}-\cos_{\Omega^\polar}\theta^\polar}}.
\]

\noindent Indeed, the function $\mathbb{H}-\cos_{\Omega^\polar}\theta^\polar$ is Lipschitz continuous and separated from $0$ in a neighborhood of $t_0$.

\begin{figure}[ht]
	\centering
	\includegraphics[width=0.5\textwidth]{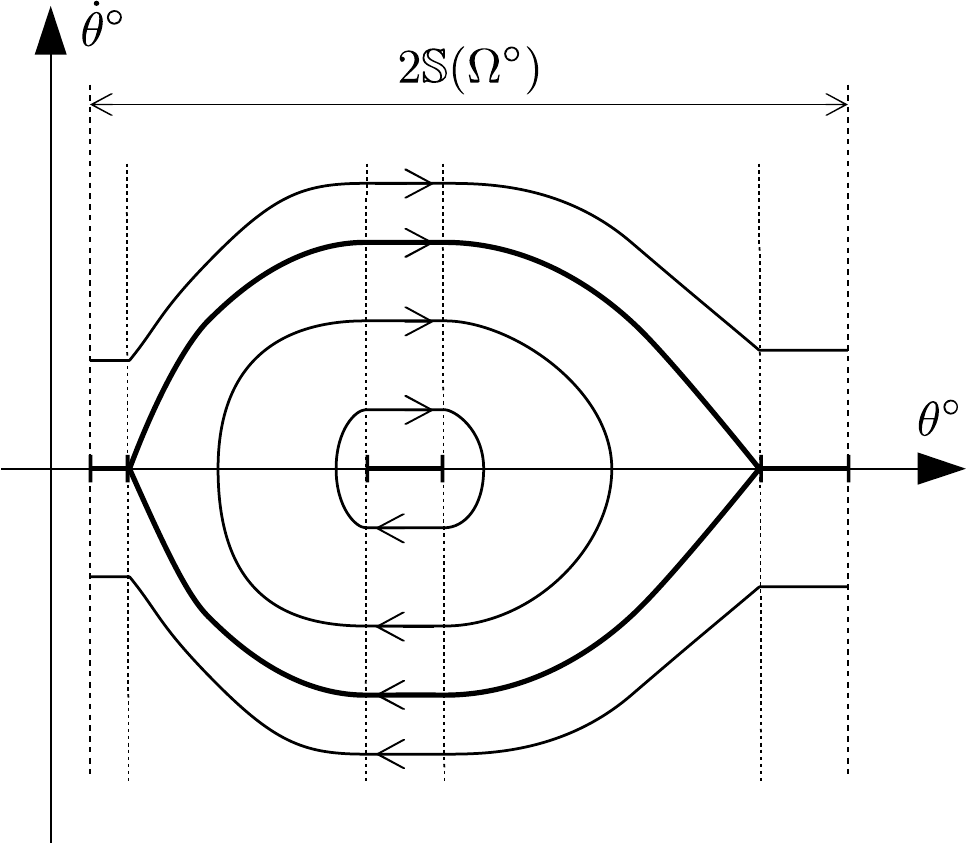}
	\caption{The phase portrait of the generalized pendulum equation~\eqref{eq:pendulum}.}
	\label{fig:phase_portrait}
\end{figure}

Let us describe the topological structure of the phase portrait on the plane  $\R^2=\{(\theta^\polar,\dot\theta^\polar)\}$ (see Fig.~\ref{fig:phase_portrait}). Obviously, the phase portrait is symmetrical w.r.t. to the reflection $\dot\theta^\polar\mapsto-\dot\theta^\polar$, if we simultaneously change the direction of motion, $t\mapsto-t$.

Similarly to the ordinary pendulum, we have
\[
	\cos_{\Omega^\polar}\theta^\polar=\mathbb{H}-\frac12{{}\dot\theta^\polar}^2\le \mathbb{H}.
\]

\noindent Thereby the energy levels of $\mathbb{H}$ are convenient to depict by the vertical lines $p_1=\mathbb{H}$ in the plane $\R^2=\{(p_1,p_2)\}$ containing $\Omega^\polar$. Let us move this line parallel to itself by changing the energy level $\mathbb{H}$ from $-\infty$ to $+\infty$.

If the whole polar set  $\Omega^\polar$ lies to the right of the vertical line $p_1=\mathbb{H}$ for some energy level $\mathbb{H}$, then there are no pendulum trajectories corresponding to this energy level.

Starting from the minimum possible energy level $\mathbb{H}=\mathbb{H}^-$ (when the vertical line $p_1=\mathbb{H}^-$ becomes supporting to the polar set $\Omega^\polar$ for the first time) we will receive either a fixed point, either an interval of fixed points depending on whether the polar set $\Omega^\polar$ has a vertical edge in the left half-plane $p_1<0$ or not.

With increasing $\mathbb{H}$ (until the line $p_1=\mathbb{H}$ becomes supporting to $\Omega^\polar$ for the second time) we will receive a family of periodic solutions. Indeed, there are no fixed points on any trajectory from the family, since a fixed point may appear only on the axis $\dot\theta^\polar=0$, but in this case $\sin_\Omega\theta\ne 0$ for any $\theta\leftrightarrow\theta^\polar$ (as the energy level $\mathbb{H}$ does not reach the extreme value). The period can be found in standard way: if $\theta_{0,1}^\polar$ are solutions of the equation $\cos_{\Omega^\polar}\theta^\polar=\mathbb{H}$ such that the function $\cos_{\Omega^\polar}\theta^\polar$ is less than $\mathbb{H}$ on the interval $(\theta^\polar_0;\theta^\polar_1)$, then
\begin{equation}
\label{eq:pendulum_period}
	\tau=\sqrt{2}\int_{\theta^\polar_0}^{\theta^\polar_1}
	\frac{d\theta^\polar}{\sqrt{\mathbb{H}-\cos_{\Omega^\polar}\theta^\polar}}.
\end{equation}

\noindent This integral is finite, since in a neighborhood of the angles $\theta^\polar_{0,1}$ we have
\[
	\cos_{\Omega^\polar}\theta^\polar = \mathbb{H} - 
	(\theta^\polar-\theta^\polar_{0,1})\sin_\Omega\theta_{0,1} + 
	o(\theta^\polar-\theta^\polar_{0,1}),
\]

\noindent where $\theta_{0,1}=\lim_{\theta^\polar\to\theta_{0,1}^\polar\pm0}\theta$. Moreover, $\sin_\Omega\theta_{0,1}\ne 0$ as the line $p_1=\mathbb{H}$ is not supporting to the polar set $\Omega^\polar$. There is no absence of uniqueness at the points $\dot\theta^\polar=0$ due to the fact that for any $\theta\leftrightarrow\theta^\polar$ the value $\sin_\Omega\theta$ is separated from $0$ in neighborhoods of these points, and consequently the trajectory passes through these points only at isolated time instants.

The vertical line $p_1=\mathbb{H}$ becomes supporting to $\Omega^\polar$ for the second time when the energy reaches some level $\mathbb{H}=\mathbb{H}^+$. We will get two types of solutions for this energy level.

\begin{enumerate}
	\item There is either a ``fixed'' point or an interval of ``fixed'' points depending on whether the polar set $\Omega^\polar$ has a vertical edge in the right half-plane $p_1>0$ or not.
	
	\item There are two separatrix trajectories emanating from one end of the interval of ``fixed'' points of type 1 to the other.
\end{enumerate}

Here we should be very careful with the term ``fixed'', since there may be an absence of uniqueness of solutions. Namely, there may be (depending on the structure of boundary of $\Omega$) a principal topological difference in the separatrix energy level $\mathbb{H}=\mathbb{H}^+$ between the phase portraits of the generalized pendulum and the classical one. Motion along one or both separatrices to the corresponding end-point may take, generally speaking, a finite time. This depends on types of singularities at the end points of integral~\eqref{eq:pendulum_period}. If the motion time is finite, then the uniqueness is lost: we may arrive at the end ``fixed'' point (if this process takes a finite time), stay at this point for some time (possibly zero, finite, or infinite), and then move out to any available separatrix (if this process takes a finite time too).

Hence there is a unique solution passing through any internal point of the right-hand edge of $\partial\Omega$ (if this edge exists). This solution is a trivial constant trajectory, $\dot\theta^\polar=0$. So internal points of the right-hand edge of $\Omega$ are fixed in the classical sense. On the other hand, there may be a lot of solutions passing through the end points of the right vertical edge of $\Omega$ (or through an extreme right-hand point of $\Omega$ if there is no edge), but a trivial constant solution also exists. Thus this is not quite right to call these points ``fixed'' in the classical sense.

Note that these motions of the generalized pendulum in the separatrix energy level generate very remarkable extremals in Martinet's problem, and the Engel and Cartan groups. We can move along a singular extremal (of the first order) on some edge, then move to a non-singular one at any instant. This non-singular extremal will come back to the singular one in finite time, and then everything will repeat for countable many times. The greatest interest here, from the author's point of view, is the question on the optimality of such piecewise singular trajectories.

It remains to consider the energy levels with $\mathbb{H}>\mathbb{H}^+$. The polar set $\Omega^\polar$ lies to the left of the line $p_1=\mathbb{H}$. Thus we receive a periodic solution (if we assume that $\theta^\polar\in\R/2\SQ(\Omega^\polar)\Z$) with the following period:
\[
	\tau=\frac{1}{\sqrt{2}}\int_0^{2\SQ(\Omega^\polar)}
	\frac{d\theta^\polar}{\sqrt{\mathbb{H}-\cos_{\Omega^\polar}\theta^\polar}}.
\]

\noindent The motion speed is always separated from $0$, since $|\dot\theta^\polar|\ge \sqrt2\sqrt{\mathbb{H}-\mathbb{H}^+}$.

Now let us integrate explicitly the pendulum equation~\eqref{eq:pendulum} for the case when $\Omega$ is a polygon. In this case solutions of equation~\eqref{eq:pendulum} is composed of solutions corresponding to edges of $\Omega^\polar$. These solutions are easy to find. Indeed, using notations of Section~\ref{sec:polyhedron}, the function $\sin_\Omega\theta$ takes the following constant value $\sin_\Omega\Theta_{k+1}$ for all $\theta^\polar\in(\Theta^\polar_k,\Theta^\polar_{k+1})$. So the integral curves are parabolas with the horizontal axis (or horizontal lines if $\sin_\Omega\Theta_{k+1}=0$) of the following form
\[
	\frac12{{}\dot\theta^\polar}^2=\theta^\polar\sin_\Omega\Theta_{k+1} + \const
	\quad\mbox{with}\quad \theta\in(\Theta^\polar_k,\Theta^\polar_{k+1}).
\]

In the polygon case the points of non-uniqueness on separatrices (on the level $\mathbb{H}=\mathbb{H}^+$) always exist in the phase portrait, since integral~\eqref{eq:pendulum_period} with $\mathbb{H}=\mathbb{H}^+$ has end singularities equivalent to $\int_0^\varepsilon t^{-\frac12}\,dt$, $\varepsilon>0$, and thus it should be finite.

Let us say a few words about the optimal control determined by the equations $u_1=\cos_\Omega\theta$ and $u_2=\sin_\Omega\theta$ for $\theta\leftrightarrow\theta^\polar$.

\begin{thm}
\label{thm:pendumul_is_pieswise_const}
	Let $\theta^\polar(t)$ be a solution of equation \eqref{eq:pendulum} and $\theta^\polar(t)\leftrightarrow\theta(t)$ for a.e.~$t$. Suppose that $\Omega$ is a convex polygon and $0\in\Int\Omega$. Then the function $\theta(t)$ is a.e.~piecewise-constant up to the period $2\SQ(\Omega)$.
\end{thm}

\begin{proof}
	Let the solution $\theta^\polar(t)$ belong to an energy level $\mathbb{H}=\const$. Obviously, $\mathbb{H}\ge\mathbb{H}^-$. First, consider the general case $\mathbb{H}>\mathbb{H}^-$ and $\mathbb{H}\ne\mathbb{H}^+$. In this case any solution $\theta^\polar(t)$ of equation~\eqref{eq:pendulum} passes through vertices of $\Omega^\polar$ only at isolated time instants. We claim that the function $\theta(t)$ is a.e.\ piecewise-constant (up to the period $2\SQ(\Omega)$). Indeed, while the point $(\cos_{\Omega^\polar}\theta^\polar,\sin_{\Omega^\polar}\theta^\polar)$ moves along an edge $Q_{k-1}Q_k$ of $\Omega^\polar$, the angle $\theta$ is determined by the point $P_\theta=(\cos_\Omega\theta,\sin_\Omega\theta)\in\partial\Omega$, which stays at the vertex $P_k$ of $\Omega$, and hence $\theta=\Theta_k+2\SQ(\Omega)\mathbb{Z}$ for a.e.~$t$.

	Consider the case $\mathbb{H}=\mathbb{H}^-$. In this case $\dot\theta^\polar\equiv0$. Consequently $\sin_\Omega\theta=0$ and $\cos_\Omega\theta<0$ for a.e.~$t$. Thus the angle $\theta$ is determined uniquely by the intersection point of the ray $\{x\le 0,y=0\}$ with the boundary $\partial\Omega$.
	
	The separatrix case $\mathbb{H}=\mathbb{H}^+$ is the most interesting one. Here we have two different types of solutions. The first type is the easiest one and it appears when the polar set $\Omega^\polar$ has a vertical edge in the right half-plane. In this case we have a family of constant solutions $\theta^\polar=\const$, which stay at arbitrary interior points of the right edge. Thus $\dot\theta^\polar=\sin_\Omega\theta=0$ and $\cos_\Omega\theta>0$. So the function $\theta$ is a.e.\ constant (up to the period) and it is determined by the intersection point of the ray $\{x\ge 0,y=0\}$ with the boundary $\partial\Omega$. The second type is more difficult: $\theta^\polar(t)$ alternates a separatrix round (which takes a fixed finite time) and a constant solution that stays at some extreme right vertex of $\Omega^\polar$ (for an arbitrary time interval). During the separatrix round the solution $\theta^\polar(t)$ passes through vertices of $\Omega^\polar$ only at isolated time instants, and $\theta(t)$ is a.e.\ piecewise-constant (up to the period). And finally, if $\theta^\polar(t)$ stays at some right vertex for $t\in[t_1;t_2]$, then we again have $\sin_\Omega\theta=0$ and $\cos_\Omega\theta>0$, and thus the angle $\theta$ is determined uniquely by the intersection point of the ray $\{x\ge 0,y=0\}$ with the boundary $\partial\Omega$.
\end{proof}

\begin{corollary}
	Let $\Omega$ be a convex polygon and $0\in\Int\Omega$. Consider the sub-Finsler problems on the Cartan an Engel groups and Martinet's case (see \S\ref{sec:Cartan} and \S\ref{sec:Martine_Engel}). Suppose $H\ne 0$. If $q\ne0$ or $h_3\ne 0$, then the optimal control is a.e.~piecewise-constant, and extremals are piecewise-polynomial of degree 3 or less. If $q=h_3=0$, then the optimal control belongs to a fixed edge of $\Omega$.
\end{corollary}

\begin{proof}
	We show in \S\ref{sec:Cartan} and \S\ref{sec:Martine_Engel} that if $H\ne0$, then the optimal control $u=(u_1,u_2)$ has the from $u=(\cos_\Omega\theta,\sin_\Omega\theta)$ where $\theta\leftrightarrow\theta^\polar$ and $\theta^\polar$ is a solution of equation \eqref{eq:generalized_pendulum_H} (up to a rotation of $\Omega$). Without loss of generality assume that $q\ge 0$. If $q\ne 0$, then $\theta^\polar$ is a solution of equation \eqref{eq:pendulum} up to stretching of time. Consequently, if $q\ne 0$, then the optimal control is piecewise-constant by Theorem~\ref{thm:pendumul_is_pieswise_const}. Hence we trivially obtain that extremals are piecewise-polynomial of degree 3 or less. If $q=0$, then we have equations similar to Heisenberg's case, which was considered in details in \S\ref{sec:Heisenberg}.
\end{proof}

Note that if $\Omega$ is a polygon, then equations from the maximum principle on the Cartan and Engel groups and in Martinet's case can be integrated explicitly (except for the case $q=h_3=0$, which leads to an additional problem with one-dimensional control), extremals are piecewise-polynomial, and any optimal control is piecewise-constant. In the case $\mathbb{H}\ne\mathbb{H}^\pm$ any optimal control is periodic and it takes values only at vertices of~$\Omega$. There is a singular control in each case $\mathbb{H}=\mathbb{H}^\pm$, and this control can be distinct from the vertices of $\Omega$. In the case $\mathbb{H}=\mathbb{H}^-$ the optimal control is singular and constant. In the case $\mathbb{H}=\mathbb{H}^+$ the optimal control is piecewise-constant, it alternates singular and non-singular types, and it can be non-periodic.

\begin{figure}[ht]
	\centering
	\begin{subfigure}[b]{0.35\textwidth}
		\includegraphics[width=\textwidth]{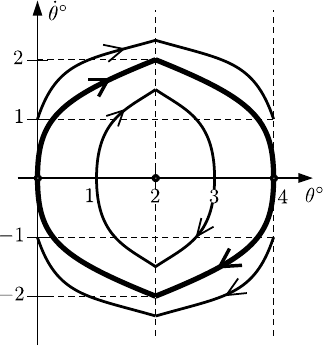}
	\end{subfigure}
	\ \ \ 
	\begin{subfigure}[b]{0.35\textwidth}
		\includegraphics[width=\textwidth]{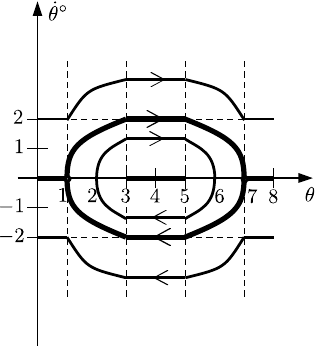}
	\end{subfigure}
	\caption{Phase portraits of the equation $\ddot\theta^\polar=\sin_\Omega\theta$ for the following sets: $\Omega^\polar=\{|p|+|q|\le 1\}$ (on the left) and $\Omega^\polar=\{|p|\le 1,|q|\le 1\}$ (on the right).}
	\label{fig:phase_pendulum_squares}
\end{figure}

\begin{example}
	Let $\Omega$ be one of the two squares  $\{|p|\le 1,|q|\le 1\}$ or $\{|p|+|q|\le 1\}$. Then the function $\sin_\Omega\theta$ takes only values $0$ and $\pm 1$ while the point $(\cos_{\Omega^\polar}\theta^\polar,\sin_{\Omega^\polar}\theta^\polar)$ moves along an edge of $\Omega$. Consequently parabolas in phase portraits of equation~\eqref{eq:pendulum} for these two sets are given by the following formulae:

	\[
		\theta^\polar = \pm\frac12{{}\dot\theta^\polar}^2 + \const.
	\]

	\noindent Phase portraits are schematically depicted in Fig.~\ref{fig:phase_pendulum_squares}.
\end{example}

The author expresses deep gratitude to Yu.L.~Sachkov and A.A.~Ardentov for valuable discussions and many interesting considerations.

\printbibliography[heading=bibintoc]


\end{document}